\documentclass[12pt]{amsart}
\usepackage{graphicx}
\usepackage[centertags]{amsmath}
\usepackage{amsfonts}
\usepackage{amssymb}
\usepackage{amsthm}
\newtheorem{thm}{Theorem}

\newtheorem{lem}{Lemma}[section]
\newtheorem{cor}{Corollary}

\newtheorem{prop}[lem]{Proposition}
\theoremstyle{definition}

\theoremstyle{remark}
\newtheorem{rem}{Remark}[section]
\numberwithin{equation}{section}

\newcommand{\norm}[1]{\left\Vert#1\right\Vert}

\newcommand{\set}[1]{\left\{#1\right\}}

\newcommand{\pd}[2]{\frac{\partial #1}{\partial #2}}

\newcommand{\vphi}{{\varphi}}

\newcommand{\G}{\Gamma}

\newcommand{\calC}{\mathcal{C}}

\newcommand{\calF}{\mathcal{F}}

\newcommand{\calL}{\mathcal{L}}

\newcommand{\calO}{\mathcal{O}}

\newcommand{\calS}{\mathcal{S}}

\newcommand{\bbZ}{\mathbb{Z}}

\newcommand{\bbQ}{\mathbb{Q}}
\newcommand{\bbR}{\mathbb R}
\newcommand{\bbC}{\mathbb C}

\newcommand{\bbN}{\mathbb N}

\newcommand{\bbH}{\mathbb{H}}

\newcommand{\Tr}{\operatorname{tr}}

\newcommand{\SO}{\mathrm{SO}}

\newcommand{\PSL}{\mathrm{PSL}}

\newcommand{\vol}{\mathrm{vol}}

\newcommand{\lap}{\triangle}
\newcommand{\bs}{\backslash}
\newcommand{\id}{1\!\!1}

\newcommand{\sgn}{\mathrm{sgn}}
\renewcommand{\Im}{\mathrm{Im}}

\begin{document}
\title[Hybrid trace formula]
{Hybrid trace formula for a non-uniform irreducible lattice in $\PSL_2(\bbR)^n$}%
\author{Dubi Kelmer}%
\address{Department of Mathematics, Boston College,  }
\email{kelmer@bc.edu}

\thanks{This work was partially supported by NSF grant DMS-1237412.}%
\subjclass{}%
\keywords{}%

\date{\today}%
\dedicatory{}%
\commby{}%
\begin{abstract} In a series of lectures Selberg introduced a trace formula on the space of hybrid Maass-modular forms of an irreducible uniform lattice in $\PSL_2(\bbR)^n$. In this paper we derive the analogous formula for a non-uniform lattice and use it to study the distribution of elliptic-hyperbolic conjugacy classes. In particular, for Hilbert modular groups there is a nice interpretation of these results in terms of class numbers and fundamental units of quadratic forms over totally real number fields.
\end{abstract}

\maketitle

\section*{Introduction}
Let $\Gamma$ denote a lattice in $\PSL_2(\bbR)$. For any hyperbolic $\gamma\in \Gamma$ (i.e., with $|\Tr(\gamma)|>2$) let $\rho_\gamma>1$ such that $\gamma\sim \left(\begin{smallmatrix} \sqrt{\rho_\gamma} & 0\\0 & 1/\sqrt{\rho_\gamma}\end{smallmatrix}\right)$. The Selberg Zeta function of $\Gamma$ is then defined by the Euler product
\[Z(s,\Gamma)=\prod_{\{\gamma\}\in \Gamma_{\mathrm{h}}'}\prod_{k=0}^\infty (1-\rho_\gamma^{-s+k}),\]
where the first product is over all primitive (i.e., not a power of another element) hyperbolic conjugacy classes. This product converges for $\Re(s)>1$, and using the Selberg trace formula Selberg \cite{Selberg56} showed that $Z(s,\Gamma)$ has an analytic continuation to $\bbC$ and satisfies a functional equation relating $Z(s,\Gamma)$ and $Z(1-s,\Gamma)$. Furthermore, its zeros in the critical strip $0\leq\Re(s)\leq 1$ are located on the line $\{\tfrac{1}{2}+i\bbR\}\cup(0,1]$ and are related to the spectrum of the hyperbolic Laplacian %$\lap=y^2(\pd{}{x^2}+\pd{}{y^2})$
on $L^2(\Gamma\bs \bbH)$. Here $\bbH$ denotes the upper half plane endowed with the hyperbolic metric and $\Gamma$ acts by isometries via linear fractional transformations.% $\gamma z=\frac{az+b}{cz+d}$ for $\gamma=\left(\begin{smallmatrix} a & b\\ c& d\end{smallmatrix}\right)$.

Now, let $\Gamma\subset \PSL_2(\bbR)^n$ denote an irreducible uniform lattice. In his 1995 lecture \cite{Selberg95}, Selberg described how to construct partial Zeta functions
from the primitive elliptic-hyperbolic conjugacy classes in $\Gamma$. Here, an element $\gamma\in \Gamma$ is called elliptic-hyperbolic if $|\Tr(\gamma_j)|<2$ for $j=1,\ldots, n-1$ and $|\Tr(\gamma_n)|>2$. For an elliptic-hyperbolic $\gamma\in \Gamma$, let $\rho_{\gamma}>1$ correspond to $\gamma_n$ as above and for $j=1,\ldots, n-1$ let $\gamma_j\sim \left(\begin{smallmatrix} \epsilon_{\gamma_j} & 0\\0 & \bar{\epsilon}_{\gamma_j}\end{smallmatrix}\right)$ with $\epsilon_{\gamma_j}\in S^1$.
For any weight $m\in \bbN^{n-1}$, the corresponding partial Zeta function is given by
\begin{equation*}
Z_{m}(s,\Gamma)=\prod_{\{\gamma\}\in \Gamma_{\mathrm{eh}}'}\mathop{\prod_{k_n\geq 0}}_{|k_j|<m_j}\left(1-\epsilon_{\gamma_1}^{k_1}\cdots \epsilon_{\gamma_{n-1}}^{k_{n-1}}\rho_\gamma^{-s-k_n}\right)^{(-1)^{n-1}},
\end{equation*}
where the notation $\Gamma_{\mathrm{eh}}'$ indicates that we only consider primitive elliptic-hyperbolic classes.

Just as in the classical setting of the Selberg Zeta function, the analytic properties of these partial Zeta functions are derived from a corresponding trace formula. Specifically, this is a trace formula on a space of hybrid Hilbert Maass-modular forms that we will now describe: Consider a product, $\bbH^n$, of $n$ hyperbolic planes. For any $m\in \bbZ^n$ let $L^2(\Gamma\bs \bbH^n,m)$ denote the space of functions $\psi$ on $\bbH^n$ which are square integrable on $\Gamma\bs \bbH^n$ and satisfy $\psi(\gamma z)=j_\gamma(z)^{m}\psi(z)$ for all $\gamma\in\Gamma$. Here we used the shorthand
$$j_\gamma(z)^{m}=j_{\gamma_1}(z_1)^{m_1}j_{\gamma_2}(z_2)^{m_2}\cdots j_{\gamma_n}(z_n)^{m_n},$$
where $j_{\gamma_j}(z_j)=\frac{cz_j+d}{c\bar{z}_j+d}$ for $\gamma_j=\left(\begin{smallmatrix} a &b\\c &d\end{smallmatrix}\right)$. Let $\lap_{z_j,m_j}$ denote the partial $m_j$-Laplacian acting on the $j$'th coordinate,
\[\lap_{z_j,m_j}=y_j^2(\pd{}{x_j^2}+\pd{}{y_j^2})-2im_jy_j\pd{}{x_j}.\]
For $m\in \bbN^{n-1}$ let $M_{m}(\Gamma)$ denote the subspace of $L^2(\Gamma\bs \bbH^n,m)$ composed of joint eigenfunctions of $\lap_{z_j,m_j},\;j=1,\ldots, n-1$ with corresponding eigenvalues $|m_j|(1-|m_j|)$. Here and below we naturally identify $\bbN^{n-1}\subset \bbZ^{n-1}\subset\bbZ^n$ by making the last coordinate zero. We similarly define the spaces $M_{\sigma m}(\Gamma)$ for all possible signs $\sigma\in \{\pm1\}^{n-1}$.

For  $m\in \bbN^{n-1}$ and $\sigma\in\{\pm 1\}^{n-1}$ let $\{\psi_k^{(\sigma m)}\}_{k\in \bbN}$ be an orthonormal basis for $M_{\sigma m}(\Gamma)$ composed of eigenfunction of $\lap_{z_n}$ with eigenvalues
$$0<\lambda_0(\sigma m)<\lambda_1(\sigma m)\leq \lambda_3(\sigma m) \ldots.$$
\begin{rem}
Recall that for a lattice $\Gamma$ in $\PSL_2(\bbR)$, the Maass wave forms are the Laplace eigenfunctions in $L^2(\Gamma\bs \bbH)$, and that the space of modular forms of weight $2m$ can be identified with the subspace of $L^2(\Gamma\bs \bbH,m)$ composed of eigenfunctions of $\lap_m$ with eigenvalue $m(1-m)$ (via the map $f(z)\mapsto y^mf(z)$). Correspondingly, we can think of the functions $\psi^{(m)}_k$ as hybrid forms, behaving like Hilbert modular forms of weight $m$ in the first $n-1$ coordinates, and like Maass forms in the last coordinate.
\end{rem}

For $\Gamma\subset \PSL_2(\bbR)^n$ an irreducible uniform torsion free lattice, the trace formula for these hybrid forms takes the following form
(see \cite[Theorem 7']{Kelmer10Holonomy}).
\begin{eqnarray*}
\lefteqn{\sum_{\sigma\in\{\pm 1\}^{n-1}}\sum_{k=1}^\infty h(r_{k,\sigma m}) +(-2)^{n-1}\delta_{m,1}h(\tfrac{i}{2})}\\
&&=\frac{|m|^* \vol(\Gamma\bs \bbH^n)}{2(2\pi)^n}\int_{\bbR} h(r)r\tanh(\pi r)dr\\
&+& (-1)^{n-1}\sum_{\{\gamma\}\in \Gamma_{\mathrm{eh}}'}\sum_{l=1}^\infty \log(\rho_\gamma)\frac{\hat{h}(l\log(\rho_\gamma))}{\rho_\gamma^{l/2}-\rho_\gamma^{-{l/2}}}\prod_{j=1}^{n-1}\left(\sum_{|k|<|m_j|}\epsilon_{\gamma_j}^{kl}\right)
\end{eqnarray*}
where $h(r)$ is any even holomorphic function with Fourier transform $\hat{h}$ compactly supported, $\lambda_k(m)=\frac{1}{4}+r_{k,m}^2$, and $|m|^*=\prod_{j=1}^{n-1} (2|m_j|-1)$.
\begin{rem}
The no torsion assumption is only used to simplify the exposition, when $\Gamma$ has torsion there is a similar formula with an additional term corresponding to the elliptic classes. The assumption on the compact support of $\hat{h}$ can also be relaxed. The sum over all possible signs $\sigma\in \{\pm1\}^{n-1}$ is also not necessary, and there is a similar formula without it. We chose to present the formula in this form to emphasize the resemblance to the logarithmic derivative of the partial Zeta function
\[\frac{Z_m'(s,\Gamma)}{Z_m(s,\Gamma)}=(-1)^{n-1}\sum_{\gamma\in \Gamma_{\mathrm{eh}}'}\sum_{l=1}^\infty \log(\rho_\gamma)\frac{\rho_\gamma^{(1/2-s)l}}{\rho_\gamma^{l/2}-\rho_\gamma^{-l/2}}\prod_{j=1}^{n-1}\left(\sum_{|k|<|m_j|}\epsilon_{\gamma_j}^{kl}\right).\]
\end{rem}

From this trace formula we get, by standard arguments, that $Z_m(s,\Gamma)$ has an analytic continuation and functional equation relating $s,1-s$.
The zeros of $Z_m(s,\Gamma)$ are the trivial zeros located on the negative integers, and the spectral zeros located at $s\in \bbC$ such that $s(1-s)=\lambda_k(\sigma m)$ with order equal to $\#\{k,\sigma|s(1-s)=\lambda_k(\sigma m)\}$. In the special case where all $|m_j|=1$ and $n$ is odd (respectively even) there is another zero (respectively pole) at $s=1$ of order $2^{n-1}$.

\begin{rem}
The partial Zeta function is actually the square of an analytic function defined by taking the product over just one out of each pair of inverse classes (note that the contribution of $\{\gamma\}$ and $\{\gamma^{-1}\}$ to the Zeta function is the same). Moreover, under certain assumption on the lattice (specifically, if it is derived from a maximal order in a quaternion algebra), it follows from \cite[Corollary 6.1]{Kelmer10Holonomy} that $Z_m(s,\Gamma)$ is actually a $2^{n-1}$ power of an analytic function, which explains the order of the zero/pole at $s=1$.
\end{rem}

In addition to the application to the partial Zeta functions, the hybrid trace formula was used in \cite{KelmerSarnak09} in order to establish the existence of a strong spectral gap for $\G\bs\PSL_2(\bbR)^n$. In particular, the bounds on the strong spectral gap imply that $\lambda_0(m)\geq c(\Gamma)>0$ for all $m\in \bbZ^{n-1}$; this can be interpreted as a uniform zero free region for all the partial zeta functions $Z_m(s,\Gamma),\; m\in \bbN^{n-1}$. The hybrid trace formula was also used in \cite{Kelmer10Holonomy} to compute the asymptotics of the number of elliptic-hyperbolic classes (with $\rho_\gamma$ bounded), and to study the distribution of the corresponding elliptic partes $\epsilon_{\gamma_j}$ in $S^1$.

Perhaps the most famous examples of irreducible lattices in $\PSL_2(\bbR)^n$ are the Hilbert modular groups,
$\Gamma_F=\PSL_2(\calO_F)$, where $\calO_F$ denotes the ring of integers of a totally real number field of degree $[F:\bbQ]=n$. These lattices are non-uniform (that is, the quotient $\Gamma_F\bs\bbH^n$ is not compact), and hence \cite[Theorem 7']{Kelmer10Holonomy} does not apply directly. Nevertheless, it was shown in  \cite{Kelmer10Holonomy} that when $n>2$ there is a companion uniform lattice $\tilde{\Gamma}_F$, such that the geometric side of the hybrid trace formulas for $\Gamma_F$ and $\tilde{\Gamma}_F$ are identical. Using this observation, the results of \cite{Kelmer10Holonomy} on the distribution of elliptic-hyperbolic classes, as well as the properties of the partial Zeta functions still hold for $\Gamma_F$ as long as $[F:\bbQ]>2$.

If we replace $\Gamma_F$ with a finite index subgroup it is less clear how to find such a companion uniform lattice, moreover, for quadratic fields such a companion uniform lattice does not exist even for the full group. In order to deal with these cases, we need to use the hybrid trace formula for an irreducible non-uniform lattice $\Gamma\subseteq \PSL_2(\bbR)^n$.

In this case, there is an orthogonal decomposition
$$L^2(\Gamma\bs \bbH^n,m)=L^2_c(\Gamma\bs\bbH^n,m)\oplus L^2_d(\Gamma\bs \bbH^n,m),$$
into the continuous spectrum (exhausted by the Eisenstein series) and the discrete spectrum (spanned by the cusp forms and residual forms).
We then define $M_m(\Gamma)\subseteq L^2_d(\Gamma\bs \bbH^n,m)$ as above (using only the discrete part of the spectrum).
In order to derive the hybrid trace formula here, we need to account for the continuous spectrum, the parabolic elements, and the hyperbolic elements having a parabolic fixed point, and their contributions to the trace formula.
We will show that when $n>2$ all these new contributions cancel out perfectly and we recover exactly the same formula as in the uniform case.
%\begin{thm}
%Let $\Gamma\subseteq \PSL_2(\bbR)^n$ denote an irreducible torsion free lattice.When $n>2$ the formula as in \cite[Theorem 7']{Kelmer10Holonomy} above holds also for nonuniform lattices.
%\end{thm}
(This cancelation is not very surprising in view of the results for Hilbert modular groups). The case of $n=2$ is more interesting as some of the new contributions do not cancel out and we get terms that did not appear in the uniform case. Specifically, in this case we have
\begin{thm}\label{t:htrace2}
Let $\Gamma\subseteq \PSL_2(\bbR)^2$ denote an irreducible torsion free lattice with $\kappa$ inequivalent cusps. For any even holomorphic function $h(r)$ with Fourier transform $\hat{h}$ compactly supported, and any $m\in \bbN$
\begin{eqnarray*}
\lefteqn{\sum_{k=1}^\infty h(r_{k,m}) -\delta_{m,1}h(\tfrac{i}{2})
=\tfrac{(2m-1)\vol(\Gamma\bs \bbH^2)}{(4\pi)^2}\int_{\bbR} h(r)r\tanh(\pi r)dr}\\
&&-\frac{1}{2}\sum_{\gamma\in \Gamma_{\mathrm{eh}}'}\sum_{l=1}^\infty \log(\rho_\gamma)\frac{\hat{h}(l\log(\rho_\gamma))}{\rho_\gamma^{l/2}-\rho_\gamma^{-l/2}}\bigg(\sum_{|k|<m}\epsilon_{\gamma}^{\ell k}\bigg)\\
&&-\sum_{i=1}^\kappa R^{(i)}\bigg(\hat{h}(0)+\sum_{l=1}^\infty 2\hat{h}(2lR^{(i)})\exp(-2lR^{(i)}(m-\tfrac{1}{2}))\bigg)
\end{eqnarray*}
where $R^{(i)}>0$ denotes the regulator of the $i$'th cusp.
\end{thm}
\begin{rem}
The assumption of no torsion is only used to simplify the exposition and the assumption on the compact support of $\hat{h}$ can be relaxed; see Theorem \ref{t:htrace} below for the general formula. Note that the spectrum of $\lap_{z_2}$ on $M_m(\Gamma)$ and $M_{-m}(\Gamma)$ is the same so there is no need to sum over these two sign changes.
\end{rem}

The analytic properties of the partial Zeta functions $Z_m(s,\Gamma)$ for a non-uniform irreducible lattice $\Gamma\subseteq \PSL_2(\bbR)^n$ follow from the trace formula as in the uniform case. The only difference is when $n=2$, in which case $Z_m(s,\Gamma)$ has additional double zeros at $1-m+\frac{\pi k}{R^{(i)}}i,\; k\in \bbZ$ corresponding to the additional terms coming from the cusps.

\begin{rem}
While working on this paper we learned that Gon \cite{Gon12} also derived this formula for the Hilbert modular group $\PSL_2(\calO_F)$ with $F$ a quadratic field of class number one. He then used it to study the corresponding partial Zeta functions and counting function of elliptic-hyperbolic elements in this setting.
We note that his definition of the Zeta functions is slightly different from ours.
\end{rem}

We conclude the introduction by presenting a few consequences regarding the counting function of elliptic-hyperbolic classes and the distribution of their elliptic part for a non-uniform lattice. The following is an improvement of \cite[Corollary 6.2]{Kelmer10Holonomy}
\begin{thm}\label{t:equi}
For $\Gamma\subseteq \PSL_2(\bbR)^n$ a non-uniform irreducible lattice and a smooth function, $f$, on $(S^1)^{n-1}$
\[\mathop{\sum_{\{\gamma\}\in \Gamma_{\mathrm{eh}}'}}_{\rho_\gamma\leq T} f(\epsilon_\gamma)=2^{n-1}\mathrm{Li}(T)\mu^{n-1}(f)+O_f(T^{3/4}),\]
where $\mu$ is the measure on $S^1\cong \bbR/2\pi\bbZ$ given by $d\mu(\theta)=\sin^2(\tfrac{\theta}{2})\frac{d\theta}{\pi}$, and $\mathrm{Li}(T)=\int_2^T\frac{dt}{\log(t)}$.
\end{thm}
\begin{rem}
For $\Gamma=\PSL_2(\calO_F)$ with $[F:\bbQ]>2$ this was proved in \cite[Corollary 6.2]{Kelmer10Holonomy} under the additional assumption that $f$ is even in two of its coordinates. Our new result removes this assumption and is also valid when $F$ is a quadratic extension and $\Gamma$ an arbitrary irreducible lattice.
\end{rem}

Paying a small price in the exponent of the error term, we can also allow for a sharp cutoff.
\begin{cor}\label{c:equi}
There is $\delta=\delta_k>0$ such that for any $k$ arcs $I_1,\ldots, I_k$ in $S^1$
\[\frac{\#\{\{\gamma\}\in \Gamma_{\mathrm{eh}}'|\rho_\gamma\leq T,\; \epsilon_{\gamma_j}\in I_j\}}{\#\{\{\gamma\}\in \Gamma_{\mathrm{eh}}'|\rho_\gamma\leq T\}}=\mu(I_1)\cdots\mu( I_k)+O(T^{-\delta}),\]
where the implied constant does not depend on the arcs.
\end{cor}
\begin{rem}
Assuming the Selberg-Ramanujan conjecture we can take $\delta_k=\frac{1}{2(k+2)}$. Unconditionally using the known bounds for the spectral gap, due in this setting to Blomer and Brumley \cite{BlomerBrumley10}, we can take $\delta_1=\frac{1}{6},\delta_2=\frac{1}{8}$ and $\delta_k=\tfrac{25}{64(k+1)}$ for $k>2$. Since the error term in this result does not depend on the choice of arcs $I_j$, we can also let the arcs shrink with $T$, at any rate slower than the error term of $O(T^{-\delta_k})$.
\end{rem}

For the Hilbert modular group, $\Gamma_F$, with $F$ a totally real number field, there is a nice interpretation of these results in terms of class numbers and fundamental units of quadratic forms. Fix a field embedding $\iota :F\hookrightarrow \bbR^n$ compatible with the identification of $\Gamma_F\subseteq \PSL_2(\bbR)^n$. Given a binary quadratic form $q(x,y)=ax^2+bxy+cy^2$ with $a,b,c\in \calO_F$, its discriminant is $D=b^2-4ac$, its divisor is the ideal $(a,b,c)$, its splitting field is $F(\sqrt{D})$ and its primitive discriminant is the ideal $d=\frac{(D)}{(a,b,c)^2}$. We say that two forms $q,q'$ are equivalent if $q'(x,y)=t q((x,y)\gamma)$ for some $t\in F^*$ and $\gamma\in \Gamma_F$. The splitting field and the primitive discriminant are invariant under this equivalence relation and we denote by $h(D,d)$ the number of equivalence classes with splitting field $F(\sqrt{D})$ and primitive discriminant $d$.
To any $D\in\calO_F$ with $\iota_j(D)<0$ for $j<n$ and $\iota_n(D)>0$ and an ideal $d\subset\calO_F$, there is an associated fundamental unit $\epsilon_{D,d}\in \calO_{F(\sqrt{D})}$ satisfying that $|\iota_j(\epsilon_{D,d})|^2=1$ for $j<n$ and $\iota_n(\epsilon_{D,d})>1$ (this is the fundamental solution to a certain Pellian equation; see section \ref{s:equi} for more details). We then have
\begin{cor}\label{c:units}
For any smooth function $f$ on $(S^1)^n$ that is even in each coordinate (i.e,  invariant under $\epsilon_j\mapsto \epsilon_j^{-1}$)
\[\mathop{\sum_{(D,d)}}_{\iota_n(\epsilon_{D,d})<T}  h(D,d)f(\iota_1(\epsilon_{D,d}),\ldots, \iota_{n-1}(\epsilon_{D,d}))=2^{n-2}\mu(f)\mathrm{Li}(T^2)+O_f(T^{3/2}),\]
where the sum is over discriminants $D\in \calO_F$ modulo $F^2$ that are negative in the first $n-1$ places.
\end{cor}
\begin{rem}
When $[F:\bbQ]>2$ this result already follows from \cite[Corollary 6.3]{Kelmer10Holonomy}.
\end{rem}

\section{Notation and preliminaries}\label{s:Background}
We use the notation  $A=O(B)$ or $A\lesssim B$ to indicate that $A\leq cB$
for some constant $c$, and we use subscripts (e.g., $A=O_\epsilon(B)$ or $A\lesssim_\epsilon B$) to indicate that the constant $c=c(\epsilon)$ may depend on the parameter.  We also use the notation $A(T)=o(B(T))$ to indicate that $A(T)/B(T)\to 0$ as $T\to\infty$, and $A(T)\sim B(T)$ if $A(T)/B(T)\to 1$.
\subsection{Coordinates}
Let $G=\PSL_2(\bbR)$. For any $x,t\in\bbR$ and $\theta\in [-\pi,\pi]$ let
$$n_x=\begin{pmatrix} 1 & x \\ 0& 1\end{pmatrix},\;a_t=\begin{pmatrix} e^{t/2} & 0 \\ 0& e^{-t/2}\end{pmatrix},\;k_\theta=\begin{pmatrix} \cos(\tfrac \theta 2) & \sin(\tfrac \theta 2) \\ -\sin(\tfrac \theta 2) & \cos(\tfrac \theta 2) \end{pmatrix}.$$
Any $g\in G$ has a unique decomposition $g=n_xa_tk_\theta$. In these coordinates the Haar measure of $G$ is given by $dg=e^{-t}dxdtd\theta$.

There is a natural action of $G$ on the upper half plane $\bbH=\{z=x+iy|y>0\}$ by linear fractional transformations
($g=\left(\begin{smallmatrix} a &b \\ c &d\end{smallmatrix}\right)$ sends $z$ to $\frac{az+b}{cz+d}$). The stabilizer of $i\in \bbH$ is $K=\SO(2)$ and we can identify $\bbH=G/K$. For this identification it is convenient to use the notation
$p_{z}=n_xa_{\ln(y)}=\begin{pmatrix}
\sqrt{y}  & x/ \sqrt{y}  \\
0 & 1/ \sqrt{y}
\end{pmatrix}$ for $z=x+iy\in \bbH$. For any $g\in \PSL(2,\bbR)$ we also have a decomposition $g=p_zk_\theta$, and in these coordinates we have $dg(z,\theta)=dzd\theta$, where $dz=\frac{dxdy}{y^2}$ is the hyperbolic area on $\bbH$.

\subsection{Invariant operators}
The ring of left $G$ invariant differential operators is generated by the Casimir operator, $\Omega$, given in the $z,\theta$ coordinates by (see \cite[Chapter X \S 2]{Lang85})
\begin{equation}\label{e:Casimir}\Omega=y^2(\pd{^2}{x^2}+\pd{^2}{y^2})-2y\pd{}{\theta}\pd{}{x}.\end{equation}

For any $m\in \bbZ$ let $\chi_m$ denote the characters of $K=\SO(2)$
given by $\chi_m(k_\theta)=e^{i m \theta}$.
We say that a function $\psi$ on $\PSL_2(\bbR)$ has right $K$-type $m$ (or wight $m$), if $\psi(gk)=\psi(g)\chi_m(k)$.
The ladder operators
\begin{equation}\label{e:lader}\calL^{\pm}=\pm 2iye^{\pm i\theta}(\pd{}{x}\mp i\pd{}{y})\mp 2ie^{\pm i\theta}\pd{}{\theta},\end{equation}
commute with $\Omega$ and act as raising and lowering operators between the different weights. That is, if $\psi$ is of weight $m$, then $\calL^{\pm}\psi$ is of weight $m\pm 1$.
Moreover, if $\psi$ an eigenfunction of $\Omega$ of weight $m$ and eigenvalue $\lambda$ then
(by \cite[Chapter VI \S 4,5]{Lang85})
\[\calL^\pm \calL^\mp\psi=4(\pm m(1\mp m)-\lambda)\psi.\]

For $m\in \bbZ$, consider the function $\vphi_{s,m}$ on $\PSL_2(\bbR)$ defined by
\begin{equation}\label{e:vphi} \vphi_{s,m}(n_xa_tk_\theta)=e^{st}e^{im \theta}.\end{equation}
This function is of weight $m$ and is an eigenfunction of $\Omega$ with eigenvalue $s(1-s)$. We also note for future reference that it transforms under the ladder operators via
\begin{equation}\label{e:laddervphi}\calL^\pm \vphi_{s,m}=2(s\pm m)\vphi_{s,m}\end{equation}

\begin{rem}\label{r:Ktype}
We can identify a function $\psi$ of weight $m$, with a corresponding function on $\bbH$ given by $\tilde{\psi}(z)=\psi(p_z)$. The condition that $\psi(\gamma g)=\psi(g)$ for some $\gamma=\left(\begin{smallmatrix} a & b \\ c & d\end{smallmatrix}\right)\in \PSL_2(\bbR)$ is then equivalent to the condition that
$\tilde{\psi}(\gamma z)=\tilde{\psi}(z)j_\gamma(z)^m$, with $j_\gamma(z)=\frac{cz+d}{c\bar{z}+d}$.
With this identification, the Casimir operator $\Omega$ becomes the $m$-Laplace operator $\lap_m=y^2(\pd{^2}{x^2}+\pd{^2}{y^2})-2imy\pd{}{x}$, and the ladder operators become $\calL^\pm_m=\pm 2y(i\pd{}{x}\pm \pd{}{y})\pm 2m$. This identification is convenient for certain calculations as well as for comparison with some of the literature.
\end{rem}

\subsection{Products}
We now consider a product of $n>1$ copies, $G=\prod_{j=1}^n G_j$, with each $G_j=\PSL_2(\bbR)$. Any $g\in G$ is of the form $g=(g_1,\ldots,g_n)$. We use coordinates on $G$ coming from each factor as above (e.g., $k_\theta=(k_{\theta_1},\ldots, k_{\theta_n})$ for $\theta\in [-\pi,\pi]^n$, $a_t=(a_{t_1},\ldots, a_{t_n})$ for $t\in \bbR^n$ etc.).

For $m\in \bbZ^n$ let $\chi_m(k_\theta)=e^{im\cdot \theta}$. We say that a function $\psi$ on $G$ is of wight $m$ if $\psi(gk)=\psi(g)\chi_m(k)$. The ladder operators $\calL_j^{\pm}$ send functions of wight $m$ to functions of weight $m\pm e_j$ and commute with $\Omega_j$ (and clearly also with $\calL_i^\pm$ and $\Omega_i$ for all $i\neq j$). Here and below, $e_j\in \bbZ^n$ denotes the unit vector with $1$ in the $j$'th entry.

\subsection{Conjugacy classes}
Let $G=\prod_{j=1}^n G_j$ be as above, and $\Gamma\subset G$ an irreducible lattice (that is, $\Gamma$ is a discrete subgroup with $\vol(\Gamma\bs G)<\infty$ and such that the projection to any factor of $G$ is dense). %The Hilbert modular groups are examples of non-uniform irreducible lattice in $G$, and for $n\geq 2$ they are the only examples of such lattices up to commensurability and conjugation.
The $\Gamma$-conjugacy classes are classified into the following types: \emph{Hyperbolic} (when all factors are hyperbolic), \emph{Elliptic} (when all factors are elliptic), \emph{Parabolic} (when all factors are parabolic) and \emph{mixed} (when some factors are hyperbolic and some are elliptic). We call a hyperbolic element regular if it does not have the same fixed point as a parabolic element of $\Gamma$, and we call it hyperbolic-parabolic if it does.

The conjugacy classes that will eventually contribute to the hybrid trace formula are the elliptic classes, $\Gamma_{\mathrm{e}}$, and the elliptic-hyperbolic classes, $\Gamma_{\mathrm{eh}}$, which are the mixed conjugacy classes that are elliptic in the first $n-1$ places and hyperbolic in the last.
In particular, any elliptic-hyperbolic class $\gamma\in \Gamma_{\mathrm{eh}}$ is conjugated in $G$ to $(a_{\ell_\gamma},k_{\theta_\gamma})$ for some $\ell_\gamma>0$ and $\theta_\gamma\in (-\pi,\pi)^{n-1}$ (correspondingly, $\epsilon_{\gamma_j}=e^{i\theta_{\gamma,j}/2}$ and $\rho_\gamma=e^{\ell_\gamma}$). Similarly, for an elliptic class we have that $\gamma\sim k_{\theta_\gamma}$ for $\theta_\gamma\in (-\pi,\pi)^n$.
\begin{rem}
For $\epsilon\in S^1$ the matrices $\left(\begin{smallmatrix} \epsilon& 0\\ 0& \bar\epsilon\end{smallmatrix}\right)$ and $\left(\begin{smallmatrix} \bar\epsilon& 0\\ 0& \epsilon\end{smallmatrix}\right)$ are conjugated $\PSL_2(\bbC)$. However, the matrices $k_\theta$ and $k_{\theta}^{-1}$ are not conjugated in $\PSL_2(\bbR)$ and we can use the relation $\epsilon_{\gamma_j}=e^{i\theta_{\gamma,j}}$ to to define $\epsilon_{\gamma_j}$ unambiguously.
\end{rem}

\section{Point pair invariants}
We recall some of the properties of point pair invariants on the hyperbolic plane and their lifts to functions on $\PSL_2(\bbR)$.
Let $\rho$ be a smooth compactly supported function on $(0,\infty)$. For any $m\in \bbZ$ we define a function on $\bbH\times\bbH$ as in \cite[Page 386]{Hejhal83} by
\begin{equation}\label{e:pointpair}\tilde{f}(z,w)=(-1)^m\rho\left(\frac{|z-w|^2}{\Im(z)\Im(w)}\right)\left[\frac{\bar{z}-w}{z-\bar{w}}\right]^m.\end{equation}
For any $g\in \PSL_2(\bbR)$ we have $\tilde{f}(gz,gw)=j_g(z)^mf(z,w)j_g(w)^{-m}$ and we can extend $\tilde{f}$ to a function $f$ on $\PSL_2(\bbR)\times \PSL_2(\bbR)$ by setting
$$f(p_zk_1,p_wk_2)=\chi_m(k_1k_2^{-1})\tilde{f}(z,w).$$
We call a function obtained in this manner a wight $m$ point pair invariant.
Such a point pair invariant defines an integral operator
\[L_f\psi(g)=\int_G f(g,g')\psi(g')dg',\]
sending functions of wight $m$ to functions of wight $m$ and vanishing on functions of other $K$-types.
Moreover, any $\Omega$-eigenfunction $\psi$ of wight $m$ and eigenvalue $\lambda=\frac{1}{4}+r^2$ is also an eigenfunction of $L_f$ with eigenvalue $\calS_mf(r)$, where $\calS_m(f)$ denotes the Selberg transform of $f$ given by
\begin{eqnarray}\label{e:ST}\calS_mf(r)&=&\int_{\bbH} \tilde{f}(i,z)\Im(z)^{\frac{1}{2}+ir}dz\\
\nonumber&=& \int_G f(1,g)\vphi_{\frac{1}{2}+ir,m}(g)dg\end{eqnarray}
(see \cite[pages 364 and 392]{Hejhal76}).

The functions $h(r)=\calS_mf(r)$ obtained in this way are in the Paley-Wiener space $\mathrm{PW}^w(\bbC)$ of even holomorphic functions of uniform exponential type (that is, the Fourier transform $\hat{h}$ is even smooth and compactly supported). Moreover, given $h\in \mathrm{PW}^w(\bbC)$, for any $m\in \bbZ^n$ there is a (unique) weight $m$ point pair invariant satisfying $\calS_mf=h$.
Specifically, given $h\in \mathrm{PW}^w(\bbC)$ let
$\hat{h}(u)=\frac{1}{2\pi}\int_\bbR h(r)e^{-iru}dr$ be its Fourier transform (so that $\hat{h}\in C^\infty_c(\bbR)$ and $h(r)=\int_\bbR \hat{h}(u)e^{iru}du$),
and define the auxiliary function $Q\in C^\infty_c(0,\infty)$ by $\hat{h}(u)=Q(4\sinh^2(u/2))$. We can then recover the functions $\rho$ and $Q$ from each other by the relations (see \cite[page 386]{Hejhal83})
\begin{equation}\label{e:rho}
\rho(y)=-\frac{1}{\pi}\int_\bbR Q'(y+t^2)\left[\frac{\sqrt{y+4+t^2}-t}{\sqrt{y+4+t^2}-t}\right]^mdu,
\end{equation}
\begin{equation}\label{e:Q}
Q(w)=\int_\bbR \rho(w+v^2)\left[\frac{\sqrt{w+4}+iu}{\sqrt{w+4}-iu}\right]^mdu.
\end{equation}

\subsection{Orbital integrals}
Fix $h\in \mathrm{PW}^w(\bbC)$ with Fourier transform $\hat{h}\in C^\infty_c(\bbR)$. %$\hat{h}(u)=\frac{1}{2\pi}\int_\bbR h(r)e^{-iru}dr$ be its Fourier transform (so that $\hat{h}\in C^\infty_c(\bbR)$ and $h(r)=\int_\bbR \hat{h}(u)e^{iru}du$).
For any $m\in \bbZ$, let $f_m$ denote a weight $m$ point pair invariant such that
$\calS_mf_m=h$. For $\gamma\in G$ let $G_\gamma$ denote the centralizer of $\gamma$ in $G$ and consider the orbital integral
\begin{equation}\label{e:orbital}I_m(\gamma;h)=\int_{G_\gamma\bs G}f_m(g,\gamma g)dg.\end{equation}
Here $dg$ denotes an (appropriately normalized) measure on $G_\gamma\bs G$ obtained from
Haar measure. Note that $I_m(\gamma;h)$ only depends on $\gamma$'s $G$-conjugacy class. We then have
\begin{lem}\label{l:orbital}
For any $m\in \bbZ$
\begin{equation} I_m(\gamma;h)=\left\lbrace\begin{array}{lc}
\frac{-1}{4\pi}\int_\bbR \frac{e^{-mu} \hat{h}'(u)}{\sinh(u/2)}du & \gamma=1\\
\frac{\hat{h}(\ell)}{\sinh(\ell/2)}& \gamma\sim a_\ell\\
\frac{\tilde{h}(\theta,m)}{\sin(\theta/2)}& \gamma\sim k_\theta\end{array}\right.\end{equation}
where
\begin{eqnarray}\label{e:hm}
\tilde{h}(\theta,m)=\frac{i}{4}\int_\bbR \hat{h}(u)\left[\frac{e^{\frac{(2m-1)(u+i\theta)}{2}}(e^u-e^{i\theta})}{\cosh(u)-\cos(\theta)}\right]du.
\end{eqnarray}

%\begin{equation}\label{e:orbit1}I_m(1;h)=f_m(1)=\frac{-1}{4\pi}\int_\bbR \frac{\hat{h}'(u)}{\sinh(u/2)}e^{-mu}du,\end{equation}
%for hyperbolic elements
%\begin{eqnarray}\label{e:orbith}
%I_m(a_\ell;h)=\int_\bbR f(n_x,a_\ell n_{x})dx
%=\frac{\hat{h}(\ell)}{\sinh(\ell/2)},\end{eqnarray}
%and for elliptic elements
%\begin{eqnarray}\label{e:orbite}
%I_m(k_\theta;h)=2\pi\int_0^\infty f(a_t, k_\theta a_{t})\sinh(t)dt=\frac{\tilde{h}(\theta,m)}{\sin(\theta/2)}
%,\end{eqnarray}
%with
%\begin{eqnarray}\label{e:hm}
%\tilde{h}(\theta,m)=\frac{i}{4}\int_\bbR \hat{h}(u)\left[\frac{e^{\frac{(2m-1)(u+i\theta)}{2}}(e^u-e^{i\theta})}{\cosh(u)-\cos(\theta)}\right]du.
%\end{eqnarray}
In particular, $I_0(1;h)=\frac{1}{4\pi}\int_{\bbR}h(r)r\tanh(\pi r)dr$ and for $m\neq 0$ and $\sigma=\sgn(m)=\frac{m}{|m|}$ we have for the trivial class
\[I_m(1;h)-I_{m-\sigma}(1;h)=\frac{|m|-\tfrac{1}{2}}{2\pi}h(i(|m|-\tfrac{1}{2})),\]
for $\gamma\sim k_\theta$ elliptic
\[I_m(k_\theta;h)-I_{m-\sigma}(k_\theta;h)=\frac{e^{im\theta}}{1-e^{ i\sigma\theta}}h(i(|m|-\tfrac{1}{2})),\]
and for $\gamma\sim a_\ell$ hyperbolic $I_m(a_\ell;h)$ does not depend on $m$.
\end{lem}
\begin{proof}
For the formula for $I_m(\gamma;h)$  see \cite[page 396]{Hejhal83} when $\gamma=1$, \cite[equation 6.19 on p. 389]{Hejhal83} when $\gamma\sim a_\ell$ and  \cite[equation 6.30 on p. 394]{Hejhal83} when $\gamma\sim k_\theta$. The dependence on $m$ for the trivial and elliptic classes then follow from these formulas by a direct computation.
\end{proof}

Two additional integral transforms that we will need are
\begin{equation}\label{e:Ip}
I_m^p(h)=\int_0^\infty f_m(1,n_x)\log(x)dx, \end{equation}
and
\begin{equation}\label{e:Ihp}
I_m^{hp}(t;h)=\int_\bbR f_m(n_x,a_t n_x)\log(1+x^2)dx.
\end{equation}
In particular we are interested on their dependence on $m$; we show
\begin{lem}\label{l:intp} For $m\in \bbZ$ with $\sigma=\sgn(m)\neq 0$,
\begin{eqnarray*}
I_m^p(h)-I_{m-\sigma}^p(h)=\frac{2|m|-1}{8\pi}\int_\bbR \frac{h(t)}{(|m|-\tfrac{1}{2})^2+t^2}dt-\frac{1}{4}h(i\tfrac{2|m|-1}{2})\end{eqnarray*}
\end{lem}
\begin{proof}
The integral \eqref{e:Ip} is computed explicitly in \cite[page 406-411]{Hejhal76} and the dependence on $m$ follows from this computation. Instead of repeating this computation here, in section \ref{s:intp} below we give another proof of this identity that uses the trace formula and does not require the evaluation of $I_m^p(h)$ explicitly.
\end{proof}

It is also possible to compute \eqref{e:Ihp} explicitly in terms of $h$. However, for our purpose the following vanishing result will suffice.
\begin{lem}\label{l:Ihp}
If $\mathrm{supp}(\hat{h})\subseteq[-a,a]$ then $I_m^{hp}(t;h)=0$ for $|t|>\tfrac{a}{2}$.
\end{lem}
\begin{proof}
Using \eqref{e:pointpair} for the point pair invariant we can write
\begin{eqnarray*}
f_m(n_x,a_t n_x)=(-1)^m\rho\left(\alpha^2(x^2+1)\right)\left[\frac{x\alpha+i\beta}{x\alpha-i\beta}\right]^m,
\end{eqnarray*}
with $\alpha=2\sinh(t)$ and $\beta=2\cosh(t)$.  The assumption that $\hat{h}(u)=Q(4\sinh^2(u/2))$ is supported on $[-a,a]$, implies that $Q(w)$ is supported on $(0,4\sinh^2(\tfrac{a}{2}))\subseteq(0,\alpha^2)$. Hence, $Q'(w)=0$ for $w>\alpha^2$ and from \eqref{e:rho} also $\rho(y)=0$ for $y>\alpha^2$, so that indeed
\begin{equation*}
(-1)^mI_m^{hp}(t;h)=\int_\bbR \rho(\alpha^2(x^2+1))\left[\frac{x\alpha+i\beta}{x\alpha-i\beta}\right]^m\log(1+x^2)dx=0.
\end{equation*}

\end{proof}
We finish this section with another integral identity that we will need (under a similar assumption on the support of $\hat{h}$).
\begin{lem}\label{l:integralcomp1}
Fix $a,b>0$, and let $\hat{h}\in C^\infty_c(\bbR)$ be even and supported on $[-a,a]$. Then
\[\frac{b}{\pi}\int_\bbR \frac{h(t)\cos(a t)}{b^2+t^2}dt=e^{-ab}h(ib).\]
\end{lem}
\begin{proof}
Substitute $h(t)\cos(a t)=\frac{1}{2}\int_\bbR(\hat{h}(u+a)+\hat{h}(u-a))e^{iut}dt$ to get
\begin{eqnarray*}
\frac{b}{\pi}\int_\bbR \frac{h(t)\cos(a t)}{b^2+t^2}dt&=& \frac{b}{2\pi}\int_\bbR\int_\bbR(\hat{h}(u+a)+\hat{h}(u-a)) \frac{e^{iut}}{b^2+t^2}dudt\\
&=& \frac{1}{2}\int_\bbR(\hat{h}(u+a)+\hat{h}(u-a)) \left(\frac{b}{\pi}\int_\bbR\frac{e^{iut}}{b^2+t^2}dt\right)du\\
\end{eqnarray*}
Using the residue theorem we get $\frac{b}{\pi}\int_\bbR\frac{e^{iut}}{b^2+t^2}dt=e^{-b|u|}$. Hence
\begin{eqnarray*}
\frac{b}{\pi}\int_\bbR \frac{h(t)\cos(a t)}{b^2+t^2}dt&=&\frac{1}{2}\int_\bbR(\hat{h}(u+a)+\hat{h}(u-a)) e^{-b|u|}du\\
&=& \int_0^\infty(\hat{h}(u+a)+\hat{h}(u-a)) e^{-bu}du\\
&=&\int_{-a}^\infty\hat{h}(u) e^{-b(u+a)}du+\int_a^\infty \hat{h}(u) e^{-b(u-a)}du
\end{eqnarray*}
From our assumption that $\hat{h}$ is supported on $[-a,a]$ the second term vanishes and the first one equals $e^{-ab}h(ib)$.
\end{proof}

\section{Eisenstein series}
We now recall the construction from \cite{Efrat87} of the Eisenstein series spanning $L_c^2(\Gamma\bs \bbH^n)$, and use the ladder operators to obtain the corresponding Eisenstein series for $L_c^2(\Gamma\bs \bbH^n,m)$.
\subsection{Eisenstein series at $\infty$}
We first describe the Eisenstein series corresponding to the cusp at $\infty=(\infty,\ldots,\infty)$. Let $\Gamma_\infty$ denote the stabilizer of $\infty$ in $\Gamma$, so that a typical element of $\Gamma_\infty$ is of the form
\[\begin{pmatrix} u & \alpha\\0 & u\end{pmatrix}=\left(\begin{pmatrix} u_1 & \alpha_1\\0 & u_1^{-1}\end{pmatrix},\ldots,\begin{pmatrix} u_n & \alpha_n\\0 & u_n^{-1}\end{pmatrix}\right),\]
with $\prod_j|u_j|=1$. %In fact, since $\Gamma$ is commensurable to $\SL_2(\calO_F)$ (after possibly conjugating $\Gamma$) we may assume that $\alpha\in \calO_F$ is an algebraic integer and $u\in \calO_F^*$ is in the group of units.

Consider the lattice $\calO\subseteq \bbR^n$ defined by
$$\calO=\{\alpha\in \bbR^{n}|\; n_\alpha\in \Gamma_\infty\},$$
and let $v=\vol(\bbR^n/\calO)$.
The group
$$U=\set{u |\exists \alpha,\;  \begin{pmatrix} u & \alpha\\ 0 & u^{-1}\end{pmatrix}\in \Gamma_\infty},$$
is a free group of rank $n-1$ and after fixing a set of generators $\epsilon_1,\ldots,\epsilon_{n-1}$ we can identify $U$ with $\bbZ^{n-1}$. Specifically, for $q\in \bbZ^{n-1}$ we denote by $u_q\in U$ the element $u_q=\epsilon_1^{q_1}\cdots\epsilon_{n-1}^{q_{n-1}}$. For $j=1,\ldots,n$ let $\epsilon_{i,j}$ denote the $j$-th coordinate of $\epsilon_i$ and consider the matrix
\begin{equation}\label{e:Matrix}
D=\begin{pmatrix}
\tfrac{1}{n} & \log|\epsilon_{1,1}|&\cdots & \log|\epsilon_{n-1,1}| \\
\vdots & \vdots &\ddots & \vdots \\
\tfrac{1}{n} & \log|\epsilon_{1,n})|&\cdots & \log|\epsilon_{n-1,n})|
\end{pmatrix}.\end{equation}
We define the regulator of the cusp at $\infty$ as $R=|\det(D)|$, and note that it does not depend on the choice of basis.

Given $q\in \bbZ^{n-1}$ let $\eta(q)=(\eta_1(q),\ldots,\eta_n(q))$ be defined by the equation
\[(s+\pi i\eta(q))D=(s,\pi iq),\]
where we use the notation $s+\pi i\eta(q)=(s+\pi i\eta_1(q),\ldots,s+\pi i\eta_n(q))$ and $(s,\pi iq)=(s,\pi iq_1,\ldots, \pi i q_{n-1})$.

For $m\in \bbZ^n$ and $q\in \bbZ^{n-1}$ consider the function
\begin{equation}\label{e:vphiq}\vphi_{s,m}(g,q)=\prod_j \vphi_{s+\pi i\eta_j(q),m_j}(g_j),\end{equation} where $\vphi_{s+\pi i\eta_j(q),m_j}$ is given in \eqref{e:vphi}.
Note that this function is of weight $m$, it is invariant under $\Gamma_\infty$ and it is a joint eigenfunction of $\Omega_j,\;j=1,\ldots,n$ with eigenvalues $s_j(1-s_j)$ where $s_j=s+\pi i\eta_j(q)$.

For $q\in \bbZ^{n-1}$, $m\in \bbZ^n$ and $s\in \bbC$ with $\Re(s)>1$, the weight $m$ Eisenstein series at $\infty$ is given by the following absolutely convergent series
\[E_m(g,s,q)=\sum_{\gamma\in \Gamma_\infty\bs \Gamma} \vphi_{s,m}(\gamma g,q).\]
For $m=0$ these are the Eisenstein series defined in \cite[Chapter II]{Efrat87}, and we can use the ladder operators (see \eqref{e:laddervphi}) to get that for any $m\in \bbZ^n$
\[\calL_j^\pm E_m(g,s,q)=2[m_j\pm (s+\pi i\eta_j(q))]E_{m\pm e_j}(g,s,q).\]

\subsection{Eisenstein series at other cusps}
Let $\xi_1,\ldots, \xi_\kappa$ denote a complete set of inequivalent cusps of $\Gamma$. For each $i=1,\ldots,\kappa$, let $\Gamma_{i}\subseteq\Gamma$ denote the subgroup fixing $\xi_i$. Let $\tau_i\in G$ send $\xi_i$ to $\infty$, then a typical element of
$\tau_i\Gamma_i\tau_i^{-1}$ is of the form $\begin{pmatrix} u & \alpha\\0 & u^{-1}\end{pmatrix}$. The set of $u$'s obtained in this way forms a multiplicative group $U^{(i)}$, and the set of $\alpha$'s with $n_\alpha\in \tau_i\Gamma_i\tau_i^{-1} $ form a lattice $\calO^{(i)}$.  By fixing a set of generators for $U^{(i)}$ we define as before the matrix $D^{(i)}$ as well as the elements $u_q^{(i)}$ and $\eta^{(i)}(q)$ satisfying $(s+\pi i\eta^{(i)}(q))D^{(i)}=(s,\pi i q)$ for any $q\in \bbZ^{n-1}$. Let $R^{(i)}=|\det(D^{(i)})|$ and $v^{(i)}=\vol(\bbR^n/\calO^{(i)})$. The group $U^{(i)}$, and hence also $R^{(i)}$, don't depend on the choice of $\tau_i$, but the lattice $\calO^{(i)}$ and volume $v^{(i)}$ do; we can (and will) choose $\tau_i$ so that the products $v^{(i)}R^{(i)}=vR$ are all the same. We can now define the Eisenstein series at the $i$'th cusp by
\[E^{(i)}_m(g,s,q)=\sum_{\gamma\in \Gamma_i\bs \Gamma} \vphi_{s,m}^{(i)}(\tau_i\gamma g,q),\]
where $\vphi_{s,m}^{(i)}$ is defined as in \eqref{e:vphiq} using $\eta^{(i)}(q)$.

\subsection{Scattering matrix}
We can write each Eisenstein series in a Fourier decomposition relative to the coordinates in each of the cusps.
The constant terms in these decompositions are of particular importance and we denote them by
\[E_m^{(i,j)}(g,s,q)=\frac{1}{v^{(j)}}\int_{\bbR^n/\calO_j} E_m^{(i)}(n_x\tau_jg,s,q)dx.\]
We recall the following results from \cite[Capter II, Section 1]{Efrat87} regarding the analytic continuation and functional equation of the Eisenstein series (for the case $m=0$).
\begin{prop}\label{p:sphericalscattaring}
The constant terms of the (spherical) Eisenstein series $E_0^{(i)}(g,s,q)$ are of the form
\begin{equation}\label{e:const}
E_0^{(i,j)}(g,s,q)=\delta_{i,j}\vphi^{(i)}_{s,0}(g,q)+\phi^{(i,j)}(s,q)\vphi^{(i)}_{1-s,0}(g,-q),\end{equation}
The functions $\phi^{(i,j)}(s,q)$ as well as the Eisenstein series, have a meromorphic continuation to $\bbC$ and satisfy the functional equation
\begin{equation}\label{e:func}E_0(g,1-s,-q)=\Phi(s,q)E_0(g,s,q),\end{equation}
where the scattering matrix $\Phi(s,q)$ is the matrix with $(i,j)$-coefficients $\phi^{(i,j)}(s,q)$, and $E_0(g,s,q)$ is the column vector with $i$-th entry $E_0^{(i)}(g,s,q)$. Moreover, the scattering matrix satisfies that $\Phi(\tfrac{1}{2}+it,q)$ is unitary and that $\Phi(s,q)\Phi(1-s,-q)=I$.
\end{prop}

We can use the ladder operators $\calL_j^\pm$ to get a similar result for the Eisenstein series for any weight.
Specifically, for $m\in \bbZ$ consider the rational function
\begin{equation}\label{e:Pm}P_m(s)=\prod_{k=1}^{|m|}\frac{k-s}{k-1+s}.\end{equation}
For $m\in\bbZ^n$ we consider the multivariable function $P_m=\prod_j P_{m_j}$ and
define the weight $m$ scattering matrix $\Phi_m$ by its entries
\begin{equation}\label{e:phim}\phi_m^{(i,j)}(s,q)=\phi^{(i,j)}(s,q)P_m(s+\pi i\eta^{(i)}(q))\end{equation}
Then, starting from $m=0$ and applying $\calL_j^{\pm}$ to \eqref{e:const} we get
\begin{cor}\label{c:scattaring}
The weight $m$ Eisenstein series satisfy the functional equation
\begin{equation}\label{e:func2}E_m(g,1-s,-q)=\Phi_m(s,q)E_m(g,s,q).\end{equation}
and its constant terms satisfy
\begin{equation}\label{e:const2}
E_m^{(i,j)}(g,s,q)=\delta_{i,j}\vphi^{(i)}_{s,m}(g,q)+\phi_m^{(i,j)}(s,q)\vphi^{(i)}_{1-s,m}(g,1-s,-q).\end{equation}
\end{cor}

\subsection{Spectral decomposition}
For $m\in \bbZ^n$ let $L^2(\Gamma\bs G,m)$ denote the subspace of $L^2(\Gamma\bs G)$ composed of functions of weight $m$, which can be naturally identified, as in remark \ref{r:Ktype}, with $L^2(\G\bs \bbH^n,m)$.
For each $m\in \bbZ^n,\; q\in\bbZ^{n-1}$ and each cusp, we have an isometry from $L^2(0,\infty)$ into $L^2(\G\bs G,m)$ defined on the dense set of smooth compactly supported functions by
$\rho\mapsto \int_0^\infty \rho(t)E_m^{(i)}(g,\tfrac{1}{2}+it,q)dt$.
Let $L^2_c(\G\bs G,m)$ denote the space spanned by the images of these isometries from all cusps and all $q\in \bbZ^{n-1}$. This space captures the continuous part of the spectrum. The discrete part $L^2_d(\G\bs G,m)$ is then defined as its orthogonal complement and it further decomposes as the finite dimensional residual space spanned by the residues of $E_m^{(i)}(g,s,0)$ with $s\in [0,1]$, and the space of cusp forms composed of functions that vanish at all cusps (see \cite[Theorm 8.4]{Efrat87}).

The operators $\Omega_j$ act on $L^2_{d}(\G\bs G,m)$ and we can decompose it as a direct sum of joint eigenspaces of $\Omega_1,\ldots,\Omega_n$
\[L^2_{d}(\G\bs G,m)=\bigoplus_{\lambda\in \Lambda(m)} V_\Gamma(m,\lambda),\]
where $\Lambda(m)\subset\bbR^n$ is the set of joint eigenvalues and $V_\Gamma(m,\lambda)$ the joint eigenspace corresponding to $\lambda=(\lambda_1,\ldots,\lambda_{n})$. We denote the dimension of each joint eigenspaces by $d(m,\lambda)=\dim V_\Gamma(m,\lambda)$.

Now for $m\in \bbZ^{n-1}$ with all $m_j\neq 0$ define the space
\[M_m(\Gamma)=\{\psi\in L^2_d(\G\bs G,m)|(\Omega_j+|m_j|(1-|m_j|))\psi=0,\;\forall j<n\},\]
where we identify $\bbZ^{n-1}\subset\bbZ^n$ by putting a zero in the last coordinate.

\section{The Hybrid trace formula}
For $m\in \bbZ^{n-1}$ with all $m_j\neq 0$ let $M_m(\Gamma)$ be as above. Let $\{\psi_{k}^{(m)}\}_{k=0}^\infty$ be an orthonormal basis for $M_m(\Gamma)$ composed of eigenfunctions of $\Omega_n$ with eigenvalues
\[0<\lambda_0(m)\leq \lambda_1(m)\leq \ldots\]
and use the parametrization $\lambda_k(m)=\frac{1}{4}+r_{k,m}^2$. We also denote by
$|m|^*=\prod_j(2|m_j|-1)$ and define the function
\[H_m(\theta)=\prod_{j=1}^{n-1}\frac{e^{im_j\theta_j}}{1-e^{i\sgn(m_j)\theta_j}}.\]
Let $h_0(r)$ be an even function that is analytic in a strip $|\Im(r)|<\tfrac{1}{2}+\delta$ and decay like $|h_0(r)|=O((1+\Re(r))^{-2-\delta})$ for some $\delta>0$. We will show
\begin{thm}\label{t:htrace}
For $n>2$
\begin{eqnarray*}\lefteqn{\sum_k h_0(r_{k,m})-(-1)^n\delta_{m,\sgn(m)}h_0(\tfrac{i}{2})}\\
&=& \frac{|m|^*\vol(\calF_\Gamma)}{(4\pi)^n}\int_\bbR h_0(r)r\tanh(\pi r)dr\\
&&+\sum_{\{\gamma\}\in \Gamma_{\mathrm{eh}}}\frac{\ell_{\gamma_0}\hat{h}_0(\ell_\gamma) }{2\sinh(\ell_\gamma/2)}H_m(\theta_\gamma)+\sum_{\{\gamma\}\in \Gamma_{\mathrm{e}}} \frac{\tilde{h}_0 (\theta_{\gamma,n},0)}{M_\gamma\sin(\theta_{\gamma,n}/2)}H_m(\theta_\gamma)\\
\end{eqnarray*}
where $\gamma_0$ is the (unique) primitive element for which $\gamma=\gamma_0^l$ with $l\in \bbN$, and $M_{\gamma}$ is the order of the centralizer of $\gamma$.% (which is finite for elliptic $\gamma$).

For $n=2$ there are additional terms coming from each cusp and the formula takes the form
\begin{eqnarray*}\lefteqn{\sum_k h_0(r_{k,m})-\delta_{|m|,1}h_0(\tfrac{i}{2})=
 \frac{|m|^*\vol(\calF_\Gamma)}{(4\pi)^2}\int_\bbR h_0(r)r\tanh(\pi r)dr}\\
&&+\sum_{\{\gamma\}\in \Gamma_{\mathrm{eh}}}\frac{\ell_{\gamma_0}\hat{h}_0(\ell_\gamma) }{2\sinh(\ell_\gamma/2)}H_m(\theta_\gamma)+\sum_{\{\gamma\}\in \Gamma_{\mathrm{e}}} \frac{\tilde{h}_0(\theta_{\gamma,2},0)}{M_\gamma\sin(\theta_{\gamma,2}/2)}H_m(\theta_{\gamma,1})\\
&&-\sum_{i=1}^\kappa\sum_{q\in \bbZ}R^{(i)}\hat{h}_0(2qR^{(i)})\exp(-2|q|R^{(i)}(|m|-\tfrac{1}{2}))
\end{eqnarray*}
where $R^{(i)}$ is the regulator of the $i$'th cusp.
\end{thm}

\begin{rem}
Using the identity
\[\frac{e^{im_j\theta_j}}{1-e^{i\theta_j}}+\frac{e^{-im_j\theta_j}}{1-e^{-i\theta_j}}=-\sum_{|k|<|m_j|}e^{ik\theta_j},\]
and taking the sum over all possible signs we get
\[\sum_{\sigma\in \{\pm1 \}^{n-1}} H_{\sigma m}(\theta)=(-1)^{n-1}\prod_{j=1}^{n-1}\big(\sum_{|k|<|m_j|}e^{ik\theta_j}\big).\]
We thus see that for $n>2$ our formula coincides with the formula for uniform lattices given in the introduction.
\end{rem}
\begin{rem}
By a suitable approximation argument (see e.g., \cite[Theorem 7.5]{Hejhal76}) it is enough to prove this formula for test functions $h\in \mathrm{PW}^w(\bbC)$ which is what we will assume from now on.
\end{rem}
\subsection{Setting up the trace formula}
We now review the setup needed for proving the hybrid trace formula.
In addition to our test function, $h_0$, we fix $n-1$ auxiliary functions $h_j\in PW^w(\bbC)$ for $j=1,\ldots,n-1$. We only require the auxiliary test functions  to satisfy that $h_j(i(|m_j|-\tfrac{1}{2}))\neq 0$ and that $\hat{h}_j$ is supported on $(-\delta,\delta)$ for some small fixed $\delta<\min_{i} R^{(i)}$ (the latter condition is used only when $n=2$ in order to simplify some computations).
To simplify notation we also set $h_n=h_0$ and consider the multivariable function $h(r)=h_1(r_1)\cdots h_n(r_n)$.

For $m\in \bbZ^{n}$ let $f_{m_j},\; j=1,\ldots, n$ denote a point-pair invariants of weight $m_j$ such that $S_{m_j}f_{m_j}=h_j$ and consider the multivariable functions $f_m(g,g')=\prod_j f_{m_j}(g_j,g_j')$.
We denote by $F_{m}$ the operator on $L^2(\Gamma\bs G,m)$ with kernel
$$F_m(g,g')=\sum_{\gamma\in \Gamma}f_m(g,\gamma g'),$$ that is
\begin{eqnarray*}
F_{m}\psi(g)&=&\int_{\calF_\Gamma} F_m(g,g')\psi(g')dg'\\
%&=&\int_G f_m(g,g')\psi(g')dg',
\end{eqnarray*}
where $\calF_\Gamma\subseteq G$ is a fundamental domain for $\Gamma\bs G$.

Let $E_m^{(i)}(g,s,q),\;i=1,\ldots,\kappa$ denote the Eisenstein series corresponding to each of the cusps, and consider the operator
$H_{m}$ on $L^2(\Gamma\bs G,m)$ given by the kernel
\begin{equation}\label{e:Hm}\end{equation}
\begin{eqnarray*}\lefteqn{H_{m}(g,g')=}\\
&&\!\!\!\!\!\!\!\!\frac{1}{2^{n+1}vR}\sum_{i=1}^\kappa\sum_{q\in \bbZ^{n-1}}\int_\bbR h(t+\pi\eta^{(i)}(q))E_m^{(i)}(g,\tfrac{1}{2}+it,q)E_m^{(i)}(g',\tfrac{1}{2}-it,-q)dt.
\end{eqnarray*}

The same argument as in the proof of \cite[Theorem 9.7]{Efrat87} shows that $F_{m}-H_{m}$ is of trace class and that
\begin{eqnarray*}\Tr(F_m-H_{m})&=&\int_{\calF_\Gamma} (F_m(g,g)-H_m(g,g))dg\\&=&\sum_{\lambda\in \Lambda(m)}d(\lambda,m)h(r_\lambda),\end{eqnarray*}
where $r_\lambda=(r_{\lambda_1},\ldots, r_{\lambda_n})$ with $\lambda_j=\frac{1}{4}+r_{\lambda_j}^2$.

For the hybrid trace formula we are only interested in the eigenvalues $\lambda$ for which $\lambda_j=|m_j|(1-|m_j|)$ for $j=1,\ldots, n-1$. More generally, for any subset $J\subseteq\{1,\ldots, n\}$ we define the corresponding set of eigenvalues
\[\Lambda(m,J)=\{\lambda\in \Lambda(m)|\forall j\in J,\; \lambda_j=|m_j|(1-|m_j|)\}.\]
We note that the eigenfunctions corresponding to eigenvalues in $\Lambda(m)\setminus\Lambda(m,J)$ can be obtained as lifts of eigenfunctions of lower weights. This was made precise in \cite[Corollary 2.1]{Kelmer10Holonomy} stating in particular
\begin{prop}\label{p:Lambda}
For $m\in \bbZ^n$, let $\sigma=\sgn(m)\in \{0,1,-1\}^n$ and $J_m=J_\sigma=\{j|m_j\neq 0\}$. Then for any function $\Psi$ on $\bbR^n$ for which the sum on the right absolutely converges we have
\begin{eqnarray*}\lefteqn{\sum_{\lambda\in \Lambda(m,J_m)}d(m,\lambda)\Psi(\lambda)+\delta_{m,\sigma}(-1)^{|J_m|}\Psi(0)}\\
&&=\mathop{\sum_{\nu\in\{0,1\}^n}}_{J_\nu\subseteq J_m} (-1)^{|J_\nu|}\sum_{\lambda\in \Lambda(m-\sigma\nu)}d(m-\sigma\nu,\lambda)\Psi(\lambda).\end{eqnarray*}
\end{prop}

In particular, when $J_m=\{1,\ldots, n-1\}$ we get
\begin{eqnarray}\label{e:altTr}
\lefteqn{\mathop{\sum_{\nu\in\{0,1\}^n}}_{J_\nu\subseteq J_m} (-1)^{|J_\nu|}\int_{\calF_\Gamma}(F_{m-\sigma\nu}(g,g)-H_{m-\sigma\nu}(g,g))dg=}\\
\nonumber&&=\mathop{\sum_{\nu\in\{0,1\}^n}}_{J_\nu\subseteq J_m} (-1)^{|J_\nu|}\sum_{\lambda\in \Lambda(m-\sigma\nu)}d(\lambda,m-\sigma\nu)h(r_\lambda)\\
\nonumber&&=\sum_{\lambda\in \Lambda(m,J_m)}d(m,\lambda)h(r_\lambda)+\delta_{m,\sigma}(-1)^{|J_m|}h(\tfrac{i}{2})\\
\nonumber&&=\left(\sum_k h_0(r_{k,m})-(-1)^n\delta_{m,\sigma}h_0(\tfrac{i}{2})\right)\left(\prod_{j=1}^{n-1}h_j(i\tfrac{|m_j|-1}{2})\right).
\end{eqnarray}
Note that, up to the factor $\prod_{j=1}^{n-1}h_j(i\tfrac{|m_j|-1}{2})$, the right hand side is exactly the spectral side of the hybrid trace formula.
In the following sections we compute the left hand side and show that it is given by the same factor times the geometric side.

\subsection{Truncated kernels}
In order to compute the geometric side of the formula we truncate the kernels as in \cite[Chapter III]{Efrat87} and then take a limit sending the cutoff parameter to infinity. To do this, let $Y_0$ be the function $Y_0(n_xa_tk)=e^{t_1+\cdots+ t_n}$ measuring the distance into the cusp at infinity.
Fix a large parameter $A>0$ and define the truncated Eisenstein series $E_{m,A}^{(i)}(g,s,q)$ given by
\[E_m^{(i)}(g,s,q)-\delta_{i,j}\vphi_{s,m}^{(j)}(\tau_jg,q)-\phi_m^{(i,j)}(s,q)\vphi^{(j)}_{1-s,-m}(\tau_jg,q),\]
if $Y_0(\tau_jg)>A$ and by $E_m^{(i)}(g,s,q)$ otherwise. We then define the truncated operator
$H_{m,A}$ by replacing $E_m^{(i)}(g,s,q)$ by $E_{m,A}^{(i)}(g,s,q)$ in \eqref{e:Hm}.

We also truncate the operator $F_m$ replacing it by the truncated kernel
\begin{eqnarray*}\lefteqn{F_{m,A}(g,g')=}\\
&&F_m(g,g')-\sum_i \id_A(Y_0(\tau_i g))\frac{1}{v^{(i)}}\sum_{q\in \bbZ^{d-1}}\int_{\bbR^n}f_m(g,\begin{pmatrix}u_q^{(i)} & x\\ 0& u_{-q}^{(i)}\end{pmatrix}g')dx,\end{eqnarray*}
where we use a smooth cutoff function $\id_A(t)=\left\{\begin{array}{cc} 1 & t>A\\ 0 & t<A-1\end{array}\right.$.

We then have that
\[\int_{\calF_\Gamma} (F_m(g,g)-H_m(g,g))dg=\lim_{A\to\infty}\int_{\calF_\Gamma}(F_{m,A}(g,g)-H_{m,A}(g,g))dg.\]
For fixed $A$, the integrals $\int_{\calF_\Gamma}F_{m,A}(g,g)dg$ and $\int_{\calF_\Gamma}H_{m,A}(g,g)dg$ converge; while both integrals blow up logarithmically as $A\to\infty$, the logarithmic terms in $\int_{\calF_\Gamma}F_{m,A}(g,g)dg$ and $\int_{\calF_\Gamma}H_{m,A}(g,g)dg$ cancel out, and the limit of the difference can be computed explicitly in terms of the test function $h$.

\subsection{Alternating sums}
We will take a slightly different approach, more suitable for computing the alternating sum
\[\mathop{\sum_{\nu\in\{0,1\}^n}}_{J_\nu\subseteq J_m} (-1)^{|J_\nu|}\int_{\calF_\Gamma}(F_{m-\sigma\nu}(g,g)-H_{m-\sigma\nu}(g,g))dg.\]
Instead of computing the difference  $\int_{\calF_\Gamma}(F_{m,A}(g,g)-H_{m,A}(g,g))dg$ and taking $A\to\infty$ we will compute separately each of the alternating sums
\begin{equation}\label{e:altFm}\mathop{\sum_{\nu\in\{0,1\}^n}}_{J_\nu\subseteq J_m} (-1)^{|J_\nu|}\int_{\calF_\Gamma}F_{m-\sigma\nu,A}(g,g)dg,
\end{equation}
and
\begin{equation}\label{e:altHm}
\mathop{\sum_{\nu\in\{0,1\}^n}}_{J_\nu\subseteq J_m} (-1)^{|J_\nu|}\int_{\calF_\Gamma}H_{m-\sigma\nu,A}(g,g))dg.\end{equation}
We will see that, in each of these sums, the logarithmic terms cancel out and the limit as $A\to\infty$ can be computed explicitly.
Moreover, in addition to the logarithmic blowup, many terms appearing in the general trace formula also cancel out, saving quite a bit of computations.
Most of this cancelation follows from the following combinatorial observation.
\begin{lem}\label{l:cancelation}
Let $\calC:\bbZ^n\to \bbC$ be any function satisfying that for some $1\leq j\leq n$, the differences
$\calC(m)-\calC(m-e_j)$ depend only on $m_j$. Then for any $m\in \bbZ^n$ with $\{j\}\subsetneq J_m$ we have
\[\mathop{\sum_{\nu\in\{0,1\}^n}}_{J_\nu\subseteq J_m} (-1)^{|J_\nu|}\calC(m-\sgn(m)\nu)=0.\]
\end{lem}
\begin{proof}
For simplicity assume that $m_{j}>0$ (the argument for $m_j<0$ is analogous). Let $\sigma=\sgn(m)\in \{0,1,-1\}^n$.
We can separate the subsets of $J_m$ as a disjoint union of two sets
\[\{J_\nu\subseteq J_m|j\in J_\nu\}\cup \{J_\nu\subseteq J_m|j\not\in J_\nu\},\]
each having the same number of elements. We can correspondingly rewrite the alternating sum as
\begin{eqnarray*}\sum_{J_\nu\subseteq J_m} (-1)^{|J_\nu|}\calC(m-\sigma\nu)
&=& \mathop{\sum_{J_\nu\subseteq J_m}}_{j\not\in J_\nu} (-1)^{|J_\nu|}(\calC(m-\sigma\nu)-\calC(m-\sigma\nu-e_{j}))\\
&=&\mathop{\sum_{J_\nu\subseteq J_m}}_{j\not\in J_\nu} (-1)^{|J_\nu|}c(m_j)\\
&=& c(m_j)\mathop{\sum_{J_\nu\subseteq J_m}}_{j\not\in J_\nu} (-1)^{|J_\nu|}
\end{eqnarray*}
Since we assumed that $|J_m|\geq 2$, there is an even number of subsets $\{J_\nu\subseteq J_m|j\not\in J_\nu\}$, half of them with even cardinality and half with odd cardinality hence $\sum_{j\not\in J_\nu\subseteq J_m} (-1)^{|J_\nu|}=0$.

\end{proof}

\subsection{Continuous contribution}
Following the same arguments as in \cite[Chapter III, Proposition 1.1]{Efrat87} we have
 \begin{prop}\label{p:cont}
For any fixed $m\in \bbZ^n$ as $A\to \infty$
\begin{eqnarray*}
\lefteqn{\int_{\calF_\Gamma} H_{m,A}(g,g)dg=}\\
&&2^{n-1}\log(A)\sum_{i=1}^\kappa R^{(i)}\sum_{q\in\bbZ^{n-1}}\hat{h}(2\log(u_q^{(i)}))+\frac{h(0)}{4}\Tr(\Phi_m(\tfrac{1}{2},0))\\
&&\!\!\!\!\!\!-\frac{1}{4\pi}\sum_{i=1}^\kappa \sum_{q\in \bbZ^{n-1}}\int_{\bbR}h(t+\pi\eta^{(i)}(q))
\Re \left[\calC^{(i)}(m)\right]dt+o(1).
\end{eqnarray*}
where
\[\calC^{(i)}(m)=\sum_{j=1}^\kappa {\phi_m^{(i,j)}}'(\tfrac{1}{2}+it,q)\phi^{(i,j)}_m(\tfrac{1}{2}-it,-q).\]
\end{prop}
The first term, involving  $\log(A)$, does not depend on $m$ and hence cancels out when taking the alternating sum.
Also, since $P_{m}(\tfrac{1}{2})=1$ we have that $\Phi_m(\tfrac{1}{2},0)=\Phi_0(\tfrac{1}{2},0)$, and hence the only dependence on $m$ is in the $\calC^{(i)}(m)$ terms. Since
$P_m(\tfrac{1}{2}+it)P_m(\tfrac{1}{2}-it)=1$ we have
\begin{eqnarray*}
\lefteqn{{\phi_m^{(i,j)}}'(\tfrac{1}{2}+it,q)\phi^{(i,j)}_m(\tfrac{1}{2}-it,-q)= {\phi^{(i,j)}}'(\tfrac{1}{2}+it,q)\phi^{(i,j)}(\tfrac{1}{2}-it,-q)}\\
&&+\frac{P_m'}{P_m}(\tfrac{1}{2}+i(\pi\eta^{(i)}(q)+t))\phi^{(i,j)}(\tfrac{1}{2}+it,q)\phi^{(i,j)}(\tfrac{1}{2}-it,-q).
\end{eqnarray*}
The first term does not depend on $m$ and for the second term,
from unitarity of $\Phi(\frac{1}{2}+it,q)$ and the fact that $\overline{\phi^{(i,j)}(s,q)}=\phi^{(i,j)}(\bar{s},-q)$ we get that
\[\sum_{j=1}^\kappa\phi^{(i,j)}(\tfrac{1}{2}+it,q)\phi^{(i,j)}(\tfrac{1}{2}-it,-q)=1.\]
We thus see that for $m_j>0$  (respectively $m_j<0$)
\begin{eqnarray*}\lefteqn{\calC^{(i)}(m)-\calC^{(i)}(m\mp e_j)}\\
&&=\frac{P_m'}{P_m}(\tfrac{1}{2}+i(\pi\eta^{(i)}(q)+t))-\frac{P_{m\mp e_j}'}{P_{m\mp e_j}}(\tfrac{1}{2}+i(\pi\eta^{(i)}(q)+t))\\
&&= \frac{1-2|m_j|}{(|m_j|-\tfrac{1}{2})^2+(\pi \eta^{(i)}(q)+t)^2}\end{eqnarray*}
depends only on $m_j$.

Hence, by Lemma \ref{l:cancelation}, when $n>2$ there is no contribution from the continuous spectrum and for $n=2$ we have
\begin{eqnarray*}\lefteqn{\int_{\calF_\Gamma}(H_{(m,0),A}(g,g)-H_{(m-1,0),A}(g,g))dg=}\\
&&\frac{1}{4\pi}\sum_{i=1}^\kappa \sum_{q\in \bbZ}\int_{\bbR}h(t+\pi\eta^{(i)}(q))
\frac{2|m|-1}{(|m|-\tfrac{1}{2})^2+(t+\pi \eta^{(i)}_1)^2}dt+o(1)
\end{eqnarray*}
In this case the groups $U^{(i)}$ are cyclic generated by some $\epsilon^{(i)}$ satisfying that $|\epsilon_1^{(i)}\epsilon_2^{(i)}|=1$.
Hence $D^{(i)}=\begin{pmatrix} \tfrac{1}{2}& \log|\epsilon^{(i)}|\\ \tfrac{1}{2} & -\log|\epsilon^{(i)}|\end{pmatrix}$, $R^{(i)}=|\log|\epsilon^{(i)}||$, and
$\eta^{(i)}(q)=(\frac{q}{2R^{(i)}},-\frac{q}{2R^{(i)}})$. We can thus rewrite the above term as
\begin{equation}\label{e:cont1}
\frac{1}{4\pi}\sum_{i=1}^\kappa \sum_{q\in \bbZ}\int_{\bbR}h_0(t+\frac{\pi q}{2R^{(i)}}) h_1(t-\frac{\pi q}{2R^{(i)}})
\frac{2|m|-1}{(|m|-\tfrac{1}{2})^2+(t-\frac{\pi q}{2R^{(i)}})^2}dt
\end{equation}
 Making a change of variables $t\mapsto t+\frac{\pi q}{2R^{(i)}}$ and changing the order of summation and integration \eqref{e:cont1}
 becomes
 \begin{equation}\label{e:cont2}
\frac{2|m|-1}{4\pi}\sum_{i=1}^\kappa \int_{\bbR}\left(\sum_{q\in \bbZ}h_0(t+\frac{\pi q}{R^{(i)}})\right)
\frac{ h_1(t)}{(|m|-\tfrac{1}{2})^2+t^2}dt
\end{equation}
Using Poisson summation we replace
\[\sum_{q\in \bbZ}h_0(t+\frac{\pi q}{R^{(i)}})=2R^{(i)}\sum_{q\in \bbZ}\hat{h}_0(2qR^{(i)})e^{2iqR^{(i)}t},\]
and \eqref{e:cont2} becomes
 \begin{eqnarray*}
\lefteqn{\frac{2|m|-1}{4\pi}\sum_{i=1}^\kappa 2R^{(i)}\sum_{q\in \bbZ}\hat{h}_0(2qR^{(i)})\int_{\bbR}
\frac{h_1(t)e^{2iqR^{(i)}t}}{(|m|-\tfrac{1}{2})^2+t^2}dt}\\
&=&\frac{2|m|-1}{4\pi}\sum_{i=1}^\kappa 2R^{(i)}\hat{h}_0(0)\int_{\bbR} \frac{h_1(t)}{(|m|-\tfrac{1}{2})^2+t^2}dt\\
&&+\sum_{i=1}^\kappa 2R^{(i)}\sum_{q\in \bbN}\hat{h}_0(2qR^{(i)})\frac{|m|-\tfrac{1}{2}}{\pi}\int_{\bbR}\frac{h_1(t)\cos(2qR^{(i)}t)}{(|m|-\tfrac{1}{2})^2+t^2}dt
\end{eqnarray*}
Finally, using the assumption on the support of $\hat{h}_1$ we can use Lemma \ref{l:integralcomp1} to replace
\[\frac{|m|-\tfrac{1}{2}}{\pi}\int_{\bbR} \frac{h_1(t)\cos(2qR^{(i)}t)}{(|m|-\tfrac{1}{2})^2+t^2}dt=e^{-qR^{(i)}(2|m|-1)}h_1(i\tfrac{2|m|-1}{2}).\]
We can thus conclude
\begin{prop}\label{p:altcont}
For $n>2$, as $A\to  \infty$,
$$\sum_{J_\nu\subseteq J_m}(-1)^{|J_\nu|}\int_{\calF_\Gamma} H_{m,A}(g,g)dg\to 0,$$
and for $n=2$
\begin{eqnarray*}
\lefteqn{\int_{\calF_\Gamma}(H_{(m,0),A}(g,g)-H_{(m-1,0),A}(g,g))dg=}\\
&&\left(\sum_{i=1}^\kappa R^{(i)}\right)\frac{|m|-\tfrac{1}{2}}{\pi}\hat{h}_0(0)\int_{\bbR}\frac{ h_1(t)}{(|m|-\tfrac{1}{2})^2+t^2}dt\\
&&+\sum_{i=1}^\kappa 2R^{(i)}\sum_{q\in \bbN}\hat{h}_0(2qR^{(i)})e^{-2qR^{(i)}(|m|-\tfrac{1}{2})}h_1(i\tfrac{2|m|-1}{2}))+o(1)
\end{eqnarray*}
\end{prop}

\subsection{Compact contribution}
Instead of using the truncated kernel $F_{m,A}$ to evaluate \eqref{e:altFm}, it is convenient to use the original kernel, $F_m$, but integrate over a truncated fundamental domain. Just as for the $m=0$ case (see \cite[page 86]{Efrat87}) as $A\to\infty$
\[\int_{\calF_\Gamma}F_{m,A}(g,g)dg=\int_{\calF_A} F_{m}(g,g)dg+o(1),\]
where
\begin{equation}\calF_{A}=\{g\in \calF_\Gamma| Y_0(\tau_jg)<A,\; j=1,\ldots,\kappa\}\end{equation}
To compute the term on the right expand $F_m$ as a sum over all lattice elements and collect together $\Gamma$-conjugacy classes to get
\begin{eqnarray*}
\int_{\calF_A} F_{m}(g,g)dg&=&\sum_{\gamma\in \Gamma}\int_{\calF_A} f_{m}(g,\gamma g)dg\\
&=& \sum_{\{\gamma\}}\int_{\calF_{\gamma,A}}f_m(g,\gamma g)dg,
\end{eqnarray*}
where the last sum is over all $\Gamma$-conjugacy classes, and
\[\calF_{\gamma,A}=\cup_{\tau\in\Gamma/\Gamma_{\gamma}}\tau\calF_{A}.\]

We now describe the contribution of the regular conjugacy classes, that is, the trivial class and the elliptic, mixed and hyperbolic conjugacy classes that do not have a parabolic fixed point (the contribution of parabolic classes and hyperbolic-parabolic classes will be dealt it in the following sections).

For the regular classes the orbital integrals converge and we have
\[\lim_{A\to\infty}\int_{\calF_{\gamma,A}}f_m(g,\gamma g)dg=\int_{\calF_{\gamma}}f_m(g,\gamma g)dg=\vol(\Gamma_\gamma\bs G_\gamma) I_{m}(\gamma;h),\]
where $I_m(\gamma;h)=\prod_j I_{m_j}(\gamma_j;h_j)$ with $I_{m_j}(\gamma_j;h_j)$ the orbital integral given in \eqref{e:orbital}.
Using Lemma \ref{l:orbital} for these orbital integrals and taking the alternating sum over all neighboring weights we see that
for the trivial element
\begin{eqnarray*}\lefteqn{\mathop{\sum_{\nu\in\{0,1\}^n}}_{J_\nu\subseteq J_m} (-1)^{|J_\nu|}I_{m-\sigma\nu}(1;h)=}\\
&&\frac{|m|^*}{(4\pi)^n}\int_\bbR h_0(r)r\tanh(\pi r)dr\left(\prod_{j=1}^{n-1} h_j(\tfrac{2|m_j|-1}{2}i)\right).
\end{eqnarray*}
For the mixed classes (and regular hyperbolic classes) the alternating sum vanishes unless $\gamma$ is elliptic-hyperbolic in which case
\begin{eqnarray*}\lefteqn{\mathop{\sum_{\nu\in\{0,1\}^n}}_{J_\nu\subseteq J_m} (-1)^{|J_\nu|}I_{m-\sigma\nu}(\gamma;h)=}\\
&&\frac{\hat{h}_0(\ell_\gamma)}{2\sinh(\ell_\gamma/2)}H_m(\theta_\gamma)\left(\prod_{j=1}^{n-1}h_j(\tfrac{|m_j|-1}{2}i)\right),\end{eqnarray*}
or elliptic, in which case
\begin{eqnarray*}\lefteqn{\mathop{\sum_{\nu\in\{0,1\}^n}}_{J_\nu\subseteq J_m} (-1)^{|J_\nu|}I_{m-\sigma\nu}(\gamma;h)=}\\
&&\frac{\tilde{h}_0(\theta_{\gamma,n},0)}{\sin(\theta_{\gamma,n}/2)}\left(\prod_{j=1}^{n-1}  \frac{e^{im_j\theta_{\gamma,j}}}{1-e^{i\sigma_j\theta_{\gamma,j}}}\right)\left(\prod_{j=1}^{n-1} h_j(i\tfrac{2|m_j|-1}{2})\right).\end{eqnarray*}
We also note that for $\gamma$ elliptic-hyperbolic, $\Gamma_\gamma$ is cyclic generated by some primitive element $\gamma_0$ and $\vol(\Gamma_\gamma\bs G_\gamma)=\ell_{\gamma_0}$. For $\gamma$ elliptic, $\Gamma_\gamma$ is a finite group of order $M_\gamma$ and $\vol(\Gamma_\gamma\bs G_\gamma)=\frac{1}{M_{\gamma}}$.

To conclude, we have shown that
\begin{prop}
The contribution of the regular conjugacy classes to \eqref{e:altFm} is given by
\begin{eqnarray*}
\bigg(\frac{|m|^*\vol(\Gamma\bs G)}{(4\pi)^n}\int_\bbR h_0(r)r\tanh(\pi r)dr
+\sum_{\{\gamma\}\in \Gamma_{\mathrm{eh}}}\frac{\ell_{\gamma_0}\hat{h}_0(\ell_\gamma)}{2\sinh(\ell_\gamma/2)}H_m(\theta_\gamma)\\
+\sum_{\{\gamma\}\in\Gamma_{\mathrm{e}}}\frac{\tilde{h}_0(\theta_{\gamma,n},0)}{M_\gamma\sin(\theta_{\gamma,n}/2)}H_m(\theta_\gamma)\bigg)\left(\prod_{j=1}^{n-1}h_j(i\tfrac{2|m_j|-1}{2})\right)
\end{eqnarray*}
\end{prop}

\subsection{Parabolic contribution}
We can collect together the parabolic elements corresponding to each cusp and compute their contribution separately. To simplify notation we will assume for now that we are dealing with the cusp at infinity and omit the superscripts enumerating the cusps.

We recall \cite[Chapter III, Lemma 2.1]{Efrat87} stating that we can take the elements
$n_\alpha=\begin{pmatrix} 1 & \alpha \\ 0& 1\end{pmatrix}$ with $\alpha\in \calO/U^2$ as a full set of representatives for the parabolic conjugacy classes corresponding to the cusp at $\infty$. The contribution of these parabolic elements to $\Tr(F_{m,A})$ is then given by
\begin{eqnarray*}
\sum_{\alpha\in \calO/U^2}\int_{\calF_{\infty,A}} f_m(g,n_\alpha g)dg
\end{eqnarray*}
where
$$\calF_{\infty,A}=\{n_xa_tk_\theta|x\in \calF_\calO,\; Y_0(g)<A\}.$$
Indeed, the compact support of $f_{m_j}(g,1)$ implies that $f_m(g,n_\alpha g)=0$ when $g$ is close to any of the other cusps. We can thus use the domain $\calF_{\infty,A}$ ignoring the cutoff at all the other cusps.
Following the same argument as in \cite[pages 88-89]{Efrat87} we get
\begin{prop}\label{p:parabolic}
For any $m\in \bbZ^n$, as $A\to\infty$, the contribution to $\Tr(F_{m,A})$ of the parabolic classes corresponding to the cusp at infinity is given by
\begin{eqnarray*}
\lefteqn{\sum_{\alpha\in \calO/U^2}\int_{\calF_{\infty,A}} f_m(g,n_\alpha g)dg=}
\\&&(2^{n-1}R\log(A)+ c_\calO v)\hat{h}(0)+2^{n}R\hat{h}(0)\sum_{j=1}^n\frac{I_{m_j}^p(h_j)}{\hat{h}_j(0)}+o(1),
\end{eqnarray*}
where $I_{m_j}^p(h_j)$ is given in \eqref{e:Ip} and $c_\calO$ is the Euler constant of $\calO$.
\end{prop}
\begin{proof} Using the coordinates $g=n_xa_tk$ and the fact that
$$f_m(n_xa_tk,n_\alpha n_xa_tk)=f_m(1,a_{-t}n_\alpha a_t),$$
we get
\begin{eqnarray*}
\lefteqn{\sum_{\alpha\in \calO/U^2}\int_{\calF_{\infty,A}} f_m(g,n_\alpha g)dg=}\\
&&v\sum_{\alpha\in \calO/U^2}\int_{Y_0(a_t)\leq A} f_m(1, n_{\alpha e^{-t}})e^{-t}dt\\
&&=v\sum_{\alpha\in \calO/U^2}\frac{1}{N(\alpha)}\int_{N(x)\geq \frac{N(\alpha)}{A}} f_m(1, n_{x})dx
\end{eqnarray*}
where $N(\alpha)=|\alpha_1\cdots \alpha_n|,\; N(x)=|x_1\ldots x_n|$ and the integral is over all $x\in R_+^n$ with $N(x)\leq N(\alpha)$.
Exchanging the order of summation and integration we get
\begin{eqnarray*}
\lefteqn{\sum_{\alpha\in \calO/U^2}\int_{\calF_{\infty,A}} f_m(g,n_\alpha g)dg=}\\&&
v\int_{\bbR_+^n} f_m(1, n_{x})\bigg(\mathop{\sum_{\alpha\in \calO/U^2}}_{N(\alpha)\leq A N(x)}\frac{1}{N(\alpha)}\bigg)dx
\end{eqnarray*}
Using \cite[Proposition 2.2]{Efrat87} stating that as $A\to\infty$
\[\mathop{\sum_{\alpha\in \calO/U^2}}_{N(\alpha)\leq AN(x)}\frac{1}{N(\alpha)}=\frac{2^{2n-1}R}{v}\log(AN(x))+2^nc_\calO+o(1),\]
together with the identity
\[2^n\int_{\bbR_+^n}f_m(1,n_x)dx=\int_{\bbR^n}f_m(1,n_x)dx=\hat{h}(0),\]
(which follows directly from \eqref{e:ST})
we get that
\begin{eqnarray*}
\lefteqn{\sum_{\alpha\in \calO/U^2}\int_{\calF_{\infty,A}} f_m(g,n_\alpha g)dg=}\\&&
(2^{n-1}R\log(A)+ c_\calO v)\hat{h}(0)
+2^{2n-1}R\int_{\bbR^n_+} f_m(1, n_{x})\log(N(x))dx.
\end{eqnarray*}
Finally, we can replace
\begin{eqnarray*}
2^{n-1}\int_{\bbR^n_+} f_m(1, n_{x})\log(N(x))dx=2^{n-1}\sum_{j=1}^n \int_{\bbR^n_+} f_m(1, n_{x})\log(x_j)dx\\
=\sum_{j=1}^nI^p_{m_j}(h_j)\prod_{i\neq j}\hat{h}_i(0)
\end{eqnarray*}
concluding the proof.
\end{proof}
When taking the alternating sum over all neighboring weights, the first term (involving $\log(A)$) does not depend on $m$ and hence cancels out. For the second term, let
\[\calC_p(m)=2^{n}R\hat{h}(0)\sum_{j=1}^n\frac{I_{m_j}^p(h_j)}{\hat{h}_j(0)},\]
then
\[\calC_p(m)-\calC_p(m-e_j)=2^{n}R\hat{h}(0)\frac{I^p_{m_j}(h_j)-I^p_{m_j-1}(h_j)}{\hat{h}_j(0)},\]
depends only on $m_j$. Hence, from Lemma \ref{l:cancelation} for $m=(m_1,\ldots,m_{n-1},0)$ with $n>2$ we have that
$\sum_{J_\nu\subseteq J_m}(-1)^{|J_\nu|}\calC(m-\sigma\nu)=0$,
while for $n=2$ we have
\begin{eqnarray*}
\lefteqn{\calC_p(m,0)-\calC_p(m-\sigma,0)=}\\
&&=4R\hat{h}_0(0)(I_{m}^p(h_1)-I^p_{m-\sigma}(h_1))\\
&&=R\hat{h}_0(0)\frac{2|m|-1}{2\pi}\int_\bbR \frac{h_1(t)}{(|m|-\tfrac{1}{2})^2+t^2}dt-R\hat{h}_0(0)h_1(i\tfrac{2|m|-1}{2})
\end{eqnarray*}

Repeating the same argument for each cusp and collecting the corresponding terms we can conclude
\begin{prop}\label{p:paraboloic}
When $n>2$ the parabolic classes do not contribute to \eqref{e:altFm} and for $n=2$ their contribution is given by
\[\hat{h}_0(0)\left(\sum_{i=1}^\kappa R^{(i)}\right)\left(\frac{2|m|-1}{2\pi}\int_\bbR \frac{h_1(t)}{(|m|-\tfrac{1}{2})^2+t^2}dt-h_1(i\tfrac{2|m|-1}{2})\right).\]
\end{prop}

\subsection{Hyperbolic-parabolic contribution}
Here again, we compute the contribution of the hyperbolic-parabolic classes corresponding to each cusp separately; to simplify notation we assume we are dealing with the cusp at infinity and drop the superscripts.

We recall the discussion in \cite[Chapter II Section 3]{Efrat87}, showing that any hyperbolic element fixing the cusp at $\infty$ is conjugated in $\Gamma$ to
some element of the form $\gamma_{q,\alpha}=\begin{pmatrix} u_q& \alpha\\ 0& u_{-q} \end{pmatrix}$, where as before $u_q=\epsilon_1^{q_1}\cdots \epsilon_{n-1}^{q_{n-1}}$ with $\epsilon_1,\ldots,\epsilon_{n-1}$ generators for $U$. In particular, to get a full set of representatives we can take all $q\in \bbZ^{n-1}/\{\pm 1\}$ and for any such $q$ there are finitely many values of $\alpha$ (see \cite[Chapter II Proposition 3.3]{Efrat87}).
In what follows we will calculate the contribution to \eqref{e:altFm} coming from the conjugacy class of $\gamma_{q,\alpha}$
(and show that it goes to $0$ as $A\to\infty$).
\begin{rem}
In general, when doing this there could be some over counting due to hyperbolic elements fixing two different cusps. However, since we will show that there is no contribution at all from these elements we do not have to worry about this issue.
\end{rem}

The other fixed point of $\gamma_{q,\alpha}$ (other than $\infty$) is given by $\frac{-\alpha}{u_q-u_{q}^{-1}}$.
Let $\tau\in G$ be an element sending $\frac{-\alpha}{u_q-u_{q}^{-1}}$ to $\infty$. We note that it is of the form
$\tau=\left(\begin{smallmatrix}* & *\\ \frac{u_q-u^{-1}_{q}}{\Lambda}& \frac{\alpha}{\Lambda}\end{smallmatrix}\right)$ and that
$N(\Lambda)=\prod_{j=1}^n |\Lambda_j|$ is independent of the choice of $\tau$ (see \cite[Page 92]{Efrat87}).
We also recall \cite[Chapter II, Proposition 3.1]{Efrat87} stating that the centralizer $\Gamma_{\gamma_{q,\alpha}}$ is a free commutative group of rank $n-1$,
and that the set $$L_{q,\alpha}=\{q'|\gamma_{q',\alpha'}\in  \Gamma_{\gamma_{q,\alpha}}\},$$
is a lattice in $\bbR^{n-1}$ (in fact, the map sending $\gamma_{q',\alpha'}\mapsto q$ gives an isomorphism of $\Gamma_{\gamma_{q,\alpha}}$ and $L_{q,\alpha}$).

Let $\calF_{\gamma_{q,\alpha}}$ denote a fundamental domain for $\Gamma_{\gamma_{q,\alpha}}\bs G$. Using the coordinates at the cusp  $Y_0(g),Y_1(g),\ldots,Y_{n-1}(g)$ defined as in \cite[page 51]{Efrat87} for $g=n_xa_tk_\theta$ by the equation
\[\begin{pmatrix} t_1\\ t_2\\\vdots\\
t_n\end{pmatrix}=D\begin{pmatrix} \log(Y_0)\\ 2Y_1\\\vdots\\
2Y_{n-1}\end{pmatrix},\]
this fundamental domain is given by
\[\calF_{\gamma_{q,\alpha}}=\{g=n_xa_tk|x\in\bbR^n,\; Y_0\in(0,\infty),\;(Y_1,\ldots,Y_n)\in \bbR^n/L_{q,\alpha}\}.\]
The contribution of the conjugacy class $\{\gamma_{q,\alpha}\}$ to the trace of $F_{m,A}$ is then given by (see \cite[Page 94]{Efrat87})
\[\int_{S_A}f_m(g,\gamma_{q,\alpha}g)dg,\]
where
\begin{eqnarray*}
S_A=\{g\in \calF_{\gamma_{q,\alpha}}|Y_0(g)<A,\; Y_0(\tau g)<A\}.\\
\end{eqnarray*}
Recalling the relation $\tilde{f}(z,w)=f(p_z,p_w)$ we can write this integral explicitly
\begin{eqnarray*}
\lefteqn{\int_{S_A}f_m(g,\gamma_{q,\alpha}g)dg=}\\
&=&\int_{S_A}\tilde{f}_m(z,u_q^2z+\alpha u_q)dz\\
&=&\int_{\tilde{S}_A}\tilde{f}_m(i,u_q^2 i+x(u_q^2-1)+u_q)dY_1\cdots dY_{n-1}\frac{dY_0}{Y_0}dx
\end{eqnarray*}
where we made the change of variables $x_j\mapsto \frac{(u_{q,j}^2-1)x_j+u_{q,j}\alpha_j}{(u_{q,j}^2-1)y_j}$ resulting in the new fundamental domain
\[\tilde{S}_A=\{x\in \bbR^n,\; (Y_1,\ldots,Y_n)\in \bbR^n/L_{q,\alpha}, B_A(x)\leq Y_0\leq A\},\]
with $B_A(x)=\frac{N(\Lambda)^2}{AN(u_q-u_q^{-1})^2}\prod_j(x_j^2+1)^{-1}$.
We can first integrate over $Y_1,\ldots, Y_{n-1}$ to get $|L_{q,\alpha}|=\vol(\bbR^n/L_{q,\alpha})$ and then preform the integral over $Y_0$ to get
$$\log(A)-\log(B_A(x))=2\log(A)+2\log(N(\frac{u_q-u_q^{-1}}{\Lambda}))+\sum_{j=1}^n \log(x_j^2+1).$$
We are thus left with the integral over $x\in \bbR^n$ giving
\begin{eqnarray*}
\lefteqn{\int_{S_A}f_m(g,\gamma_{q,\alpha}g)dg=}\\
&=&2|L_{q,\alpha}|\big(\log(A)+\log(N(\frac{u_q-u_q^{-1}}{\Lambda}))\big)\int_{\bbR^n}f_m(n_x,a_{2\log(|u_q|)}n_x)dx\\
&&+|L_{q,\alpha}|\sum_{j=1}^n\int_{\bbR^n}f_m(n_x,a_{2\log(|u_q|)}n_x)\log(1+x_j^2)dx\\
&=&2|L_{q,\alpha}|\big((\log(A)+\log(N(\frac{u_q-u_q^{-1}}{\Lambda}))\big)\prod_j\frac{\hat{h}_j(2\log(|u_{q,j}|))}{2\sinh(\log(|u_{q,j}|))}\\
&&+|L_{q,\alpha}|\sum_{j=1}^n\prod_{i\neq j}\frac{\hat{h}_i(2\log(|u_{q,j}|))}{2\sinh(\log(|u_{q,i}|))} I_{m_j}^{he}(2\log(|u_{q,j}|);h_j).
\end{eqnarray*}
Note that, for fixed $A$, the sum over all representatives $$\sum_{q,\alpha}|\int_{S_A} f_m(g,\gamma_{q,\alpha}g)dg|<\infty,$$
absolutely converges. We can thus take the alternating sum over the neighboring weights for each of the terms individually.
The only dependence on the weight $m$ is in the orbital integral $I_{m_j}^{he}(2\log(|u_{q,j}|);h_j)$ (which depends only on $m_j$).
Hence, Lemma \ref{l:cancelation} implies that this vanishes when taking the alternating sum whenever $n>2$. For $n=2$ we recall that $U$ is generated by one element $\epsilon$, that $R=|\log|\epsilon||$ and that $u_q=\epsilon^q$ for some $q\in \bbN$. We thus get that $I_{m}^{he}(2\log(|u_{q,2}|);h_1)=I_{m}^{he}(2qR;h_1)$ which vanishes by Lemma \ref{l:Ihp} and our assumption that $h_1$ is supported on $(-R,R)$.

\begin{rem}
We used the assumption that $\hat{h}_1$ is supported on $(-R,R)$ in two places, once (via Lemma \ref{l:integralcomp1}) when computing the contribution of the continuous spectrum, and again (via Lemma \ref{l:Ihp}) when showing that there is no contribution from the hyperbolic-parabolic elements. Without this assumption we would have gotten a contribution from the hyperbolic-parabolic elements that would have had to cancel out with additional terms added to the continuous contribution.
\end{rem}

\subsection{Collecting all the terms}
We can now collect all the terms contributing to the alternating sums in \eqref{e:altFm} and \eqref{e:altHm}.
For $n>2$ the only non-vanishing contribution comes from the regular classes, so, after taking the limit $A\to\infty$ we get
\begin{eqnarray*}\lefteqn{\mathop{\sum_{\nu\in\{0,1\}^n}}_{J_\nu\subseteq J_m} (-1)^{|J_\nu|}\int_{\calF_\Gamma}\big(F_{m-\sigma\nu}(g,g)-H_{m-\sigma\nu}(g,g)\big)dg}\\
&=&\bigg(\frac{|m|^*\vol(\Gamma\bs G)}{(4\pi)^n}\int_\bbR h_0(r)r\tanh(\pi r)dr +\sum_{\{\gamma\}\in \Gamma_{\mathrm{eh}}}\frac{\ell_{\gamma_0}\hat{h}_0(\ell_\gamma)}{2\sinh(\ell_\gamma/2)}H_m(\theta_\gamma)\\
&&+\sum_{\{\gamma\}\in\Gamma_{\mathrm{e}}}\frac{\tilde{h}_0(\theta_{\gamma,n},0)}{M_\gamma\sin(\theta_{\gamma,n}/2)}H_m(\theta_\gamma)\bigg)\left(\prod_{j=1}^{n-1}h_j(i\tfrac{2|m_j|-1}{2})\right)
\end{eqnarray*}
Equating this to \eqref{e:altTr} and dividing by $\left(\prod_{j=1}^{n-1}h_j(i\tfrac{2|m_j|-1}{2})\right)$ we get the trace formula.

When $n=2$ we also have non trivial contributions from the continuous spectrum and parabolic elements. Note that the terms involving the integral $\int_{\bbR}\frac{ h_1(t)}{(|m|-\tfrac{1}{2})^2+t^2}dt$ coming from the continuous spectrum and the parabolic elements cancel each other and we get
\begin{eqnarray*}\lefteqn{\int_{\calF_\Gamma}\big((F_{(m,0)}(g,g)-H_{(m,0)}(g,g))-(F_{(m-\sigma,0)}(g,g)-H_{(m-\sigma,0)}(g,g))\big)dg}\\
&=&\bigg(\frac{(2|m|-1)\vol(\Gamma\bs G)}{(4\pi)^2}\int_\bbR h_0(r)r\tanh(\pi r)dr\\
&& +\sum_{\{\gamma\}\in
\Gamma_{\mathrm{eh}}}\frac{\ell_{\gamma_0}\hat{h}_0(\ell_\gamma)}{2\sinh(\ell_\gamma/2)}\frac{e^{im\theta_\gamma}}{1-e^{i\sigma\theta_\gamma}}+\sum_{\{\gamma\}\in\Gamma_{\mathrm{e}}}\frac{\tilde{h}_0(\theta_{\gamma,2},0)}{M_\gamma\sin(\theta_{\gamma,2}/2)}\frac{e^{im\theta_{\gamma,1}}}{1-e^{i\sigma\theta_{\gamma,1}}}\\
&&-\sum_{i=1}^\kappa R^{(i)}\hat{h}_0(0)-2\sum_{i=1}^\kappa R^{(i)}\sum_{q\in \bbN}\hat{h}_0(2qR^{(i)})e^{-2qR^{(i)}(|m|-\tfrac{1}{2})}\bigg)h_1(i\tfrac{2|m|-1}{2})
\end{eqnarray*}
Equating to \eqref{e:altTr} and dividing by $h_1(i\tfrac{2|m|-1}{2})$ we get the trace formula for $n=2$.

\subsection{Another proof of Lemma \ref{l:intp}}\label{s:intp}
We now show how the ideas from this section, when applied to the modular group $\G=\PSL_2(\bbZ)$, give another cute proof of the identity
\begin{equation*}I_m^p(h)-I_{m-1}^p(h)=\frac{2m-1}{8\pi}\int_\bbR \frac{h(t)}{(m-\tfrac{1}{2})^2+t^2}dt-\frac{1}{4}h(i\tfrac{2m-1}{2}),\end{equation*}
for $m\in \bbN$, where $I_m^p(h)$ is given in \eqref{e:Ip}.

To do this, fix a test function $h\in \mathrm{PW}^w(\bbC)$, let $f_m$ denote a weight $m$ point pair invariant with $\calS_mf_m=h$, and let $F_m$ denote the corresponding integral operator on $L^2(\G\bs\PSL_2(\bbR),m)$. Let $H_m$ be defined as in \eqref{e:Hm} using the Eisenstein series at the single cusp at infinity.
We now compute the difference
\[\Tr(F_m-H_m)-\Tr(F_{m-1}-H_{m-1}),\]
in two different ways.

On one hand, applying the lowering operator $\calL^-$ we see that all eigenfunctions contributing to $\Tr(F_m-H_m)$ with eigenvalues $\lambda\neq m(1-m)$ will cancel with corresponding eigenfunctions contribution to $\Tr(F_{m-1}-H_{m-1})$; when computing $\Tr(F_0-H_0)$ we also get a contribution from the constant eigenfunction which does not cancel out with any of the eigenfunctions contributing to $\Tr(F_1-H_1)$. We thus get that
\[\Tr(F_m-H_m)-\Tr(F_{m-1}-H_{m-1})=h(i\tfrac{2m-1}{2})(\dim M_m(\Gamma)-\delta_{m,1}),\]
where
\[M_m(\Gamma)=\{\psi\in L^2(\G\bs \bbH,m)|(\lap_m +m(1-m))\psi=0\}.\]

On the other hand, following the same arguments as above, noting that in this case there are no hyperbolic-parabolic classes nor mixed classes %so the only  contributions are from the continuous spectrum, the trivial class, the elliptic classes, and the parabolic classes,
we get that
\begin{eqnarray*} \lefteqn{\Tr(F_m-H_m)-\Tr(F_{m-1}-H_{m-1})=\frac{1-2m}{4\pi}\int_\bbR \frac{h(t)}{(m-1/2)^2+t^2}dt}\\
&&+\bigg(\vol(\calF_\Gamma)\frac{2m-1}{4\pi}+\sum_{\{\gamma\}\in \Gamma_{\mathrm{e}} }\frac{1}{M_\gamma} \frac{e^{im\theta_\gamma}}{1-e^{i\theta_\gamma}}\bigg) h(i\tfrac{2m-1}{2})\\
&&+ 2(I_m^p(h)-I_{m-1}^p(h)),
\end{eqnarray*}
where the first term comes from the continuous spectrum, the second term comes from the trivial class and the elliptic class, and the last term comes from the parabolic elements.

Equating the two we get the identity
\begin{eqnarray*}\nonumber \lefteqn{I_m^p(h)-I_{m-1}^p(h)=\frac{2m-1}{8\pi}\int_\bbR \frac{h(t)}{(m-1/2)^2+t^2}dt+\frac{h(i\tfrac{2m-1}{2})}{2}\times}\\
&&\bigg(\dim M_m(\Gamma)-\delta_{m,1}- \vol(\calF_\Gamma)\frac{2m-1}{4\pi}-\sum_{\{\gamma\}\in \Gamma_{\mathrm{e}} }\frac{1}{M_\gamma} \frac{e^{im\theta_\gamma}}{1-e^{i\theta_\gamma}}\bigg)\\
\end{eqnarray*}

Now, for the modular group we have that $\vol(\calF_\Gamma)=\frac{\pi}{3}$ and there are $3$ elliptic conjugacy classes (one of order $2$ and two of order $3$). Hence
\[\vol(\calF_\Gamma)\frac{2m-1}{4\pi} +\sum_{\{\gamma\}\in \Gamma_{\mathrm{e}} }\frac{1}{M_\gamma} \frac{e^{im\theta_\gamma}}{1-e^{i\theta_\gamma}}=\left\lbrace
\begin{array}{cc}
[\frac{m}{6}]-\frac{1}{2} & m\equiv 1\pmod{6} \\
{}[\frac{m}{6}]+\frac{1}{2} &  m\not\equiv 1\pmod{6}
\end{array}\right..\]
Moreover, we recall that $M_m(\Gamma)$ is isomorphic to the space of cusp forms of weight $2m$ which, for the modular group, is of dimension
\[\dim(M_m(\Gamma))=\left\lbrace\begin{array}{cc}
0 & m=1\\
{}[\frac{m}{6}]-1 & m>1,m\equiv 1\pmod{6}\\
{}[\frac{m}{6}] & m>1,m\not \equiv 1\pmod{6}\\
\end{array}\right.\]
Combining the two together we get that
\[\dim(M_{m}(\Gamma))-\delta_{m,1}-\vol(\calF_\Gamma)\frac{2m-1}{4\pi} -\sum_{\{\gamma\} }\frac{1}{M_\gamma} \frac{e^{2im\theta_\gamma}}{1-e^{2i\theta_\gamma}}=-\frac{1}{2}.\]
and plugging this back in gives
\begin{eqnarray*}
I_m^p(h)-I_{m-1}^p(h)=\frac{2m-1}{8\pi}\int_\bbR \frac{h(t)}{(m-1/2)^2+t^2}dt-\frac{1}{4}h(i\tfrac{2m-1}{2}),
\end{eqnarray*}
concluding the proof.

\section{Counting and Equidistribution}\label{s:equi}
The results on counting elliptic-hyperbolic classes and on the distribution of their elliptic parts follow from Theorem \ref{t:htrace} exactly as in \cite{Kelmer10Holonomy}. Specifically, we note that the family of functions
$$\{H_{\sigma m}(\theta)|m\in \bbN^{n-1},\;\sigma\in \{\pm1\}^{n-1}\},$$
form an orthogonal basis of $L^2((\bbR/2\pi\bbZ)^{n-1},\mu^{n-1})$ and we can decompose any function in this space as
\[f(\theta)=\sum_{m,\sigma} a_f(\sigma m) H_{\sigma m}(\theta),\]
with
\[a_f(\sigma m)=\int f(\theta)\overline{H_{\sigma m}(\theta)}d\mu^{n-1}(\theta).\]
Using the same argument as in the proof of \cite[Proposition 3.3]{Kelmer10Holonomy}, and noting that for congruence groups we have the bound $\lambda_0(m)\geq \frac{1}{4}-(\frac{7}{64})^2$ for the spectral gap, we get
\begin{prop}\label{p:equi}
For any smooth $f\in C^\infty((\bbR/2\pi\bbZ)^{n-1})$
\[|{\mathop{\sum_{\{\gamma\}\in \Gamma_{\mathrm{eh}}}}_{l_\gamma\leq x}}\frac{l_{\gamma_0}f(\theta_\gamma)}{2\sinh(l_\gamma/2)}-2^{n}e^{x/2}\mu(f)|\ll \sqrt{\norm{f}_\infty C(f)}e^{x/4}+C(f) e^{7x/64},\]
where $C(f)=\sum_{\sigma,m}|m|^*|a_f(\sigma m)|$.
\end{prop}
\begin{rem}
For $n>2$ the hybrid trace formula in the nonuniform case is identical to the uniform case and the proof goes over with out any changes. When $n=2$ there are  additional terms in the nonuniform case, however, it is easy to see that these terms are negligible for these asymptotic estimates (the new terms are essentially equivalent to the contribution of a finite number of short primitive geodesics).
\end{rem}

Theorem \ref{t:equi} now follows from Proposition \ref{p:equi} by a standard summation by parts argument.
Corollary \ref{c:equi} can be obtained from Proposition \ref{p:equi} by using the Beurling-Selberg functions to approximate the
indicator function of an interval (see the proof of \cite[Corollary 3.1]{Kelmer10Holonomy} for more details).
We conclude this section by recalling the relation with quadratic forms in more detail.

Let $F$ denote a totaly real number field of degree $[F:\bbQ]=n\geq 2$ with ring of integers $\calO_F$. We fix an embedding $\iota: F\hookrightarrow \bbR^n$ and correspondingly identify $\PSL_2(\calO_F)$ as a lattice in $\PSL_2(\bbR)^n$.
Given a binary quadratic form $q(x,y)=ax^2+bxy+cy^2$ with $a,b,c\in \calO_F$, its discriminant is $D=b^2-4ac$, its divisor is the ideal $(a,b,c)$, its splitting field is $F(\sqrt{D})$ and its primitive discriminant is the ideal $d=\frac{(D)}{(a,b,c)^2}$. The splitting field and the primitive discriminant are invariant under the equivalence relation $q\sim t q\circ \gamma$ for $t\in F^*$ and $\gamma\in\Gamma_F$, and we denote by $h(D,d)$ the number of equivalence classes with splitting field $F(\sqrt{D})$ and primitive discriminant $d$.

For any non-square $D\in \calO_F$, and primitive discriminant $d\subset\calO_F$ let $\calO_{D,d}$ denote the quadratic order in $F(\sqrt{D})$ with relative discriminant $d$, that is,
\[\calO_{D,d}=\set{\frac{t+u\sqrt{D}}{2}\in \calO_{F(\sqrt{D})}: d|(u^2D)}.\]
Let $\calO_{D,d}^1$ denote the group of relative norm one elements in $\calO_{D,d}$; explicitly this is all $\epsilon=\frac{t+u\sqrt{D}}{2}\in \calO_{D,d}$ with $t^2-4D=4$. This group is naturally isomorphic to the group of automorphs of a quadratic form $q=ax^2+bxy+cy^2$ with discriminant $D$ and primitive discriminant $d$. Specifically, for  any $\epsilon=\frac{t+u\sqrt{D}}{2}\in \calO_{D,d}^1$, we have that $\gamma=\begin{pmatrix}\tfrac{t-bu}{2} & -cu\\ au & \tfrac{t+bu}{2} \end{pmatrix}$ satisfies $q=q\circ \gamma$, and moreover any automorph of $q$ is of this form (see \cite[Theorem 5.6]{Efrat87}).
The map sending the equivalence class of $q$, to the conjugacy class of the centralizer $\{\Gamma_\gamma\}$, where $\gamma\in \PSL_2(\calO_F)$ is any automorph of $q$, is a correspondence between equivalence classes of forms, and centralizers of mixed elements (see \cite[Propositoin 7.1]{Efrat87}).
Moreover, if $\iota_j(D)>0$ in exactly $k$ places, then $\calO_{D,d}/\{\pm 1\}\cong \bbZ^k$ and it corresponds to the centralizer of an element that is hyperbolic in $k$ places and elliptic in the rest (see \cite[Theorem 5.7]{Efrat87}).

In particular, for $D\in \calO_F$ with $\iota_n(D)>0$ and $\iota_j(D)<0$ for $j<n$ and ideal $d\subseteq\calO_{D,d}$, let $\epsilon_{D,d}\in \calO_{D,d}^1$ denote the unique generator of $\calO_{D,d}^1/\{\pm 1\}$ satisfying that $\iota_n(\epsilon_{D,d})>1$ (the other generators are then $\pm \epsilon_{D,d}^{\pm1}$). Since the centralizer of an elliptic-hyperbolic element is generated by two primitive elements ($\gamma$ and $\gamma^{-1}$), there is a correspondence between equivalence classes of quadratic forms with splitting field $F(\sqrt{D})$ and primitive discriminant $d$, and pairs of primitive elliptic-hyperbolic conjugacy classes $\{\gamma\},\{\gamma^{-1}\}$ with $\rho_\gamma^2=\iota_n(\epsilon_{D,d})$ and $\epsilon_{\gamma_j}=\iota_j(\epsilon_{D,d})^{\pm 1}$ for $j<n$.
Note that the choice of each sign $\iota_j(\epsilon_{D,d})^{\pm 1}$ in this correspondence in ambiguous, however, for a test function $f\in (S^1)^{n-1}$ that is invariant under $\epsilon_j\mapsto \epsilon_j^{-1}$ for all $j$ this ambiguity is irrelevant and Corollary \ref{c:units} follows from Theorem \ref{t:equi} applied to $\Gamma=\PSL_2(\calO_F)$.

%*********************************************************************************************************************************

%----------------------------------------------------------------
%GATHER{Mybib.bib}   % For Gather Purpose Only
%\bibliographystyle{alpha}%{amsplain}
%\bibliography{E:/TexFiles/Bib/Mybib}%F:/TexFiles/Bib/

%*********************************************************************

%*********************************************************************

\end{document}